\documentclass[12]{article}

\usepackage{amsmath,amssymb,amsthm,color}
\usepackage{subfigure}
\usepackage{epsfig}

\newcommand{\R}{\mathbb{R}}
\newcommand{\M}{\mathbb{M}}

\newcommand{\Sphere}{\mathbb{S}}

\newcommand{\delr}{\dfrac{r}{\delta}}
\newcommand{\lp}{\left(}
\newcommand{\rp}{\right)}

\newcommand{\calA}{\mathcal{A}}

\newtheorem{theorem}{Theorem}
\newtheorem{lemma}[theorem]{Lemma}
\newtheorem{proposition}[theorem]{Proposition}
\newtheorem{corollary}[theorem]{Corollary}

\title{Scattered Data Interpolation on Embedded Submanifolds with Restricted Positive Definite Kernels: Sobolev Error Estimates}
\author{Edward Fuselier\\
Department of Mathematics and Computer Science\\
High Point University\\
High Point, NC 27262\\
\\
Grady B. Wright
\thanks{Research supported by grant
ATM-0801309, DMS-0934581, and DMS-0540779 from the National Science Foundation.}
\\
Department of Mathematics \\
Boise State University \\
Boise, ID 83725-1555\\
}

\begin{document}
\maketitle

\begin{abstract}
In this paper we investigate the approximation properties of kernel interpolants on manifolds. The kernels we consider will be obtained by the restriction of positive definite kernels on $\R^d$, such as radial basis functions (RBFs), to a smooth, compact embedded submanifold $\M\subset \R^d$. For restricted kernels having finite smoothness, we provide a complete characterization of the native space on $\M$. After this and some preliminary setup, we present Sobolev-type error estimates for the interpolation problem.  Numerical results verifying the theory are also presented for a one-dimensional curve embedded in $\R^3$ and a two-dimensional torus. 
\end{abstract}

\let\thefootnote\relax
\let\thefootnote\relax

%%%%%%%%%%%%%%%%%%%%%%%%%%%%%%%%%%%%%%%%%%%%%%%%%%%%%%%%%%%%%%%%%%%%%%%%%%%
\section{Introduction}
%%%%%%%%%%%%%%%%%%%%%%%%%%%%%%%%%%%%%%%%%%%%%%%%%%%%%%%%%%%%%%%%%%%%%%%%%%%

Kernels have proven to be quite useful in the approximation of multivariate functions given scattered data, and perhaps the most rudimentary problem of this type is that of interpolation. Let $\Omega$ be a metric space, and let $\phi:\Omega\times\Omega\rightarrow \R$ be a function (which we will refer to as a \emph{kernel}). Given a target function $f:\Omega\rightarrow \R$ and a finite set of distinct nodes $X = \{x_1,x_2,\ldots,x_N \}\subset\Omega$, one can seek an interpolant $I_{X,\phi}f$ via shifts of $\phi$, i.e. $I_{X,\phi}f$ takes the form
\[I_{X,\phi}f=\sum_jc_j\phi(\cdot\,,x_j),\]
and satisfies $I_{X,\phi}f|_{X} = f|_{X}$.  Finding the appropriate coefficients $c_j$ is a matter of inverting the Gram matrix with entries $A_{i,j}=\phi(x_i,x_j)$, which is always theoretically possible if $\phi$ is positive definite. The kernel $\phi$ can also be used to construct more general interpolants, such as those where the data is generated from various types of linear functionals (integral data, derivative data, etc.) \cite{NarcWard1994:hermiteinterpolation,Wu1992:hermiteRBF}. However, in this paper we will focus on the traditional interpolation problem.

Although kernel approximants were initially considered with the domain being a Euclidean space or a sphere \cite{Stewart1976:PDFSurvey}, the ideas have since been generalized enough to handle functions defined on other mathematical objects. Kernels have been studied that are positive definite on various Riemannian manifolds \cite{HangelbroekNarcWard2010:ManifoldKernelsI,Narc1995:kernelsonmanifolds,LevesleyRagozin:2007RBFsHomogenousManifolds}, and there are also kernels that exploit the group structure of their underlying manifold, with domains such Lie groups, projective spaces, and motion groups \cite{FilbirErb2008:PDFsOnCompactGroups, FilbirSchmid2008:SO3Stability, ZuCastellFilbir2007:MotionGroups}. These kernels are obtained in various ways in the literature. Some are defined through manifold charts, some are acquired via a special expansion in terms of eigenfunctions of the Laplace-Beltrami operator, and others are assumed to be the Green's function for some pseudo-differential operator. Here the kernels can be highly dependent on the underlying manifold.

In all of the cases discussed above, the domains involved are finite-dimensional smooth manifolds and, according to the Whitney embedding theorem, any such manifold $\M$ can be embedded into some $\R^d$ (in fact, Nash's embedding theorem guarantees that this can be done isometrically). Positive definite kernels on $\R^d$ are easy to come by, so another, seemingly naive, way of obtaining a positive definite kernel on $\M$ is simply by the restriction of a kernel defined on the ambient space.  More precisely, given a kernel $\phi:\R^d\times\R^d\rightarrow \R$ and an embedded manifold $\M\subset\R^d$, we define
\[
\psi(\cdot\,,\cdot):= \phi(\cdot\,,\cdot)|_{\M\times \M}.
\]
If $\phi$ is positive definite, so is its restriction to $\M$, making $\psi$ well-suited for interpolation problems. Although the practical value of using such a kernel is not clear in all of the aforementioned situations, in this paper we will see that it is theoretically and numerically possible to approximate functions defined on manifolds with simple kernels defined on the ambient space.  

While the benefits of approximating with a kernel intrinsic to the manifold cannot be questioned, there are many applications for which an incomplete knowledge or manageable mathematical description of the underlying manifold may prevent the construction of the intrinsic kernel.  For example, in problems from computer aided design, graphics and imaging, and computer aided engineering the manifold may be a physically relevant geometric object (such as an airplane wing) and a scalar field (such as pressure or temperature) may need to be interpolated at arbitrary locations over the object~\cite{Alfeld:1996,Bajaj:1995,BarnhillOu:1990,Davydov:2007,Foley:1990,Neamtu:2002}.  Additionally, in problems from learning theory, data samples are from a very high dimensional space, but are usually assumed to lie on a relatively low dimensional embedded submanifold that is virtually unknown~\cite{BelkinNiyogi2004:LearningOnRiemannianManifolds}.  Finally, there has been much recent interest in approximating derivatives of scalar and vector valued quantities on manifolds in both the graphics (cf.~\cite{Clarenz:2000,Desbrun:1999,Stam:2003}) and computational fluid dynamics (cf.~\cite{Adalsteinsson:2003,Calhoun:2009,Novak:2007,Sbalzarini:2006,Schwartz:2005}) communities.  These approximations are typically used for numerically solving partial differential equations defined on manifolds (such as the surface of a biological cell or membrane) for modeling processes like advection-reaction-diffusion of chemicals or fluid flows on the surfaces.  Reconstruction of a function on the underlying manifold is typically first required to then approximate its derivative.

In all the applications referenced above, the manifold could be represented by a triangular mesh, points on an implicit function (level set), or more generally by a point cloud in $\R^3$.  Since the restricted kernel method under consideration in this study is mesh-free, it applies easily to all of these cases.  Additionally, the kernel's smoothness can be increased so that derivatives of the interpolant are well-defined everywhere on the surface.  Finally, the method is based on extrinsic coordinates, which naturally bypasses any coordinate singularities inherent to a manifold-based coordinate system, as shown for the unit sphere $\Sphere^2$ in~\cite{FlyerWright09}.

To the authors' knowledge, these restricted kernels have only been studied in the special case of $\M = \Sphere^{d-1}\subset\R^d$ \cite{NarcSunWard2007:SBFRBFconnection,zuCastellFilbir2005:RestrictedRBFs}. The kernels considered in these papers are \emph{radial} on $\R^d$, i.e. $\phi(x,y)=\phi(\|x - y\|)$, and it just so happens that the restriction of a radial function on $\R^d$ to the sphere is \emph{zonal}, meaning that it depends only on the geodesic distance between its arguments. Even on more general manifolds, when the kernel possesses such a property there are powerful tools at one's disposal - most notably one has a convenient Fourier expansion of the kernel in terms of eigenfunctions of the Laplace-Beltrami operator \cite{LevesleyRagozin2007:DensityOfZonalKernels}. Of course, in the setting we consider here the restricted kernel will not necessarily be zonal. To circumvent this and other issues, we will appeal to the variational theory of Madych and Nelson, and also to the error analysis of Narcowich \emph{et. al.} \cite{MadychNelson1988:MultivariateInterpolation, NarcWardWend:2006bernstein}.

The paper is organized as follows. In the next section we introduce the necessary preliminaries, notation for manifolds, Sobolev spaces, and other essential tools.  After that, we will characterize the native space, in terms of concrete function spaces, for a large class of kernels that have been restricted to embedded submanifolds. We will then make a brief detour into measuring the distribution of sample points on embedded submanifolds. This will poise us to present interpolation error estimates for both smooth and rough target functions, which are the main results of the paper.  We conclude with numerical results verifying the two main error estimates from the paper.

%%%%%%%%%%%%%%%%%%%%%%%%%%%%%%%%%%%%%%%%%%%%%%%%%%%%%%%%%%%%%%%%%%%%%%%%%%%
\section{Notation and Preliminaries}
%%%%%%%%%%%%%%%%%%%%%%%%%%%%%%%%%%%%%%%%%%%%%%%%%%%%%%%%%%%%%%%%%%%%%%%%%%%

We will restrict our study to smooth, connected, compact manifolds with no boundary. For the reader unfamiliar with manifolds, an excellent reference is Lee's book \cite{Lee2003:smoothmanifolds}. A $k$-dimensional manifold $\M$ is defined as a topological space $\M$ which is locally identified with $\R^k$ via a collection of smoothly compatible coordinate charts. More specifically, there is an \emph{atlas} $\calA = \{(U_j,\Psi_j)\}$ of open sets $U_j\subset \M$ whose union covers $\M$, and associated smooth one-to-one maps $\Psi_j:U_j\rightarrow \R^k$ such that any transition map $\Psi_j\circ\Psi_k^{-1}$ is a smooth map where it is defined. By refining the charts as necessary, one may assume that the image of any chart is equal to an open ball around the origin. Also, since the manifolds we consider are compact, we can obviously assume that any atlas encountered has finitely many charts.

%%%%%%%%%%%%%%%%%%%%%%%%%%%%%%%%%%%%%%%%%%%%%%%%%%%%%%%%%%%%%%%%%%%%%%%%%%%
\subsection{Embedded Submanifolds in $\R^d$}
%%%%%%%%%%%%%%%%%%%%%%%%%%%%%%%%%%%%%%%%%%%%%%%%%%%%%%%%%%%%%%%%%%%%%%%%%%%

In addition to the features described previously, the manifolds we consider throughout the paper will be embedded submanifolds of $\R^d$.  However, we warn that authors use the term ``embedding'' to mean different things in similar situations. This topic can be subtle, so a precise statement of what we consider to be an ``embedded submanifold'' is in order. We will follow the definition used in \cite[Chapter 7]{Lee2003:smoothmanifolds}, although we state it in less generality. This will require a few preliminary terms. 

Let $\M\subset \R^d$ be a smooth manifold endowed with the subspace topology. Given $x\in \M$, we will denote the tangent space of $\M$ at $x$ by $T_x \M$. If $F:\M\rightarrow \R^d$ is a smooth map, the \emph{rank} of $F$ at $x\in \M$ is the rank of the Jacobian map $F_*:T_x \M\rightarrow T_{F(x)}\R^d$. A smooth map is called an \emph{immersion} if $F_*$ is injective at each point, i.e. the rank of $F$ is equal to $\mbox{dim}\,(\M)$. We say that a manifold $\M$ is an \emph{embedded submanifold} of $\R^d$ if the inclusion map $\imath:\M\hookrightarrow\R^d$ is also an immersion. Some authors call such manifolds \emph{regular submanifolds}. 

%%%%%%%%%%%%%%%%%%%%%%%%%%%%%%%%%%%%%%%%%%%%%%%%%%%%%%%%%%%%%%%%%%%%%%%%%%%
\subsubsection{Slice Charts}
%%%%%%%%%%%%%%%%%%%%%%%%%%%%%%%%%%%%%%%%%%%%%%%%%%%%%%%%%%%%%%%%%%%%%%%%%%%

There is an equivalent, local characterization of embedded submanifolds, which uses charts that utilize the ambient space \cite[Chapter 8]{Lee2003:smoothmanifolds}.  Given an embedded $k$-dimensional submanifold $\M\subset\R^d$, there is an atlas $\widetilde{\calA}=\{(\widetilde{U}_j,\widetilde{\Psi}_j)\}$, where the sets $\widetilde{U}_j$ are open in $\R^d$ and cover $\M$, and each $\widetilde{\Psi}_j$ is a $1-1$ smooth map from $\widetilde{U}_j$ to some ball around the origin, say $B(0,r_j)$, that ``straightens'' the manifold. By this we mean that \[\widetilde{\Psi}_j(\M\cap \widetilde{U}_j) \subset B'(0,r_j),\] 
where $B'(0,r_j) = \{y\in B(0,r_j) \,\,|\,\, y_{k+1}=y_{k+2}=\ldots=y_d=0\}$ can be viewed as a copy of an open ball in $\R^k$. The charts $(\widetilde{U}_j,\widetilde{\Psi}_j)$ are sometimes referred to as \emph{slice charts}. If we define $U_j := \M\cap \widetilde{U}_j$ and $\Psi_j:=\widetilde{\Psi}_j|_{U_j}$, then $\calA = \{(\Psi_j,U_j)\}$ is an atlas for $\M$ in the usual sense. As before, one can assume that $\Psi_j:U_j\rightarrow B'(0,r_j)$ and $\widetilde{\Psi}_j:\widetilde{U}_j\rightarrow B(0,r_j)$ are bijections without any loss of generality. 

%%%%%%%%%%%%%%%%%%%%%%%%%%%%%%%%%%%%%%%%%%%%%%%%%%%%%%%%%%%%%%%%%%%%%%%%%%%
\subsubsection{Distances on Embedded Submanifolds}
%%%%%%%%%%%%%%%%%%%%%%%%%%%%%%%%%%%%%%%%%%%%%%%%%%%%%%%%%%%%%%%%%%%%%%%%%%%

If $\M$ is an embedded submanifold of $\R^d$, its topology is naturally induced by the Euclidean metric. This being the case, $\M$ automatically inherits a distance function  $d_\M:\M\times \M\rightarrow \R$. Assuming $\M$ is connected, given $x,y\in \M$ we can define the distance between $x$ and $y$ to be 
\begin{equation*}
d_\M(x,y) := \inf_{\underset{\underset{\gamma(1)=y}{\gamma(0)=x}}{\gamma:[0,1]\rightarrow \M}}\int_{0}^1\left\|\gamma'(t)\right\|\,dt,
\end{equation*}
where $\gamma$ is any piecewise smooth curve in $\M$ beginning at $x$ and ending at $y$, and $\|\gamma'(t)\|$ is the Euclidean length of the tangent vector $\gamma'(t)$. Given an $x\in \M$, we denote by $B_\M(x,r)$ the open ball of radius $r$ centered at $x$, i.e. $B_\M(x,r) = \{y\in \M \,\,\,| \,\,\,d_\M(y,x)< r\}$. In the case where the underlying manifold is $\R^d$ we omit the subscript and simply write $B(x,r)$. 

%%%%%%%%%%%%%%%%%%%%%%%%%%%%%%%%%%%%%%%%%%%%%%%%%%%%%%%%%%%%%%%%%%%%%%%%%%%
\subsubsection{Tubular Neighborhoods}
%%%%%%%%%%%%%%%%%%%%%%%%%%%%%%%%%%%%%%%%%%%%%%%%%%%%%%%%%%%%%%%%%%%%%%%%%%%

It is well known that any embedded submanifold of $\R^d$ has a tubular neighborhood, which is a neighborhood of $\M$ in $\R^d$ analogous to a tube around a curve in $3$-space. A precise definition of a tubular neighborhood would require much more notation than we need here, so we omit these details and instead present a useful consequence of its existence. The interested reader can find a full discussion of the subject in most introductory books on smooth manifolds \cite{Lee2003:smoothmanifolds}.

\begin{proposition}\label{proposition:tubularngd}
Let $\M$ be a compact, smooth embedded submanifold of $\R^d$. Then there exists a neighborhood $U_{\epsilon}(\M):=\{y\in \R^d \,|\,{\rm dist}(y,\M)<\epsilon\}$ and a canonical smooth map $\mathfrak{R}:U_\epsilon(\M)\rightarrow \M$ such that $\mathfrak{R}|_{\M}$ is the identity map on $\M$.
\end{proposition}

\noindent In this case we call $\epsilon$ the \emph{radius} of the tubular neighborhood $U_\epsilon(\M)$, and the map $\mathfrak{R}$ is called a \emph{retraction}.  We remark that by restricting the radius of the tubular neighborhood slightly we can assume that the domain of $\mathfrak{R}$ is compact.

%%%%%%%%%%%%%%%%%%%%%%%%%%%%%%%%%%%%%%%%%%%%%%%%%%%%%%%%%%%%%%%%%%%%%%%%%%%
\subsection{Sobolev Spaces}
%%%%%%%%%%%%%%%%%%%%%%%%%%%%%%%%%%%%%%%%%%%%%%%%%%%%%%%%%%%%%%%%%%%%%%%%%%%

The class of functions we will be interested in approximating are from the Sobolev spaces, which are spaces that consist of all $f\in L^{p}$ that have distributional derivatives $D^{\alpha}f$ in $L^{p}$ for all multi-indices up to some order.  For Sobolev spaces on $\R^d$, we will follow the notation of Adams \cite{Adams:1975SobolevSpaces}. Let $\Omega$ be a neighborhood in $\R^d$, $1\leq p<\infty$, and $m$ be a nonnegative integer. The associated Sobolev norms are defined via 
\begin{equation}
\|f\|_{W^{m}_{p}(\Omega)}:=\left(\sum_{|\alpha|\leq m}\|D^{\alpha}f\|^{p}_{L^{p}(\Omega)}   \right)^{1/p}.\nonumber
\end{equation}
For the case $p=\infty$ we have
\begin{equation}
\|f\|_{W^{m}_{\infty}(\Omega)}:=\max_{|\alpha|\leq m}\|D^{\alpha}f\|_{L^{\infty}(\Omega)}.\nonumber
\end{equation}
It is also possible to have Sobolev spaces of fractional order. Let $1\leq p < \infty$, $m$ be a non-negative integer, and $0<t<1$. We define the Sobolev space $W^{m+t}_{p}(\Omega)$ to be all $f$ such that the following norm is finite:
\begin{equation}
\|f\|_{W^{m+t}_{p}(\Omega)}:=\left(\|f\|_{W^{m}_{p}(\Omega)}^{p}+ \sum_{|\alpha|=k}\int_{\Omega}\int_{\Omega}\frac{\left|D^{\alpha}f(x)-D^{\alpha}f(y)\right|^{p}}{\|x-y\|_2^{d+pt}}dxdy\right)^{1/p}.\nonumber
\end{equation}

We define the Sobolev spaces on embedded submanifolds as follows. Let $\M\subset \R^d$ be a compact submanifold of dimension $k$. Let $\widetilde{\calA} = \{(\widetilde{U}_j,\widetilde{\Psi}_j)\}$ be an atlas of slice charts for $\M$, and let $\calA = \{(U_j,\Psi_j)\}$ be the associated intrinsic atlas. Now let $\{\chi_j\}$ be a partition of unity subordinate to $\{\widetilde{U}_j\}$. If $f$ is a function defined on $\M$, we have the projections $\pi_j(f):\R^k\rightarrow\R$ by
\begin{equation*}
\pi_j(f)(y) = \left\{
\begin{array}{cl}
\chi_j f (\Psi_j^{-1}(y)) & y\in B'(0,r_j) \\
0 & \mbox{otherwise.}
\end{array}\right.
\end{equation*}  
Using this construction, one can now define Sobolev spaces for $1\leq p<\infty$ and $s\geq 0$ via the norms
\[\|f\|_{W^s_p(\M)} := \left(\sum_{j=2}^N\| \pi_j(f)\|^2_{W^s_p(\R^k)}\right)^{1/2},\]
where $N$ is obviously the number of charts in the atlas. The norm for $W^s_p(\M)$ obviously depends on the particular choice of atlas $\widetilde{\calA}$ and the partition of unity. However, if one uses different collections of these objects, the same space arises and the norms are equivalent (the details for the case $d=k-1$ can be found in Lions and Magenes \cite{Lions:1961nonhomogeneous}; for general $k<d$ the argument is similar). Also, as is customary in the case $p=2$ we define $H^s(\M) := W^s_2(\M)$. 

%%%%%%%%%%%%%%%%%%%%%%%%%%%%%%%%%%%%%%%%%%%%%%%%%%%%%%%%%%%%%%%%%%%%%%%%%%%
\subsection{Traces}
%%%%%%%%%%%%%%%%%%%%%%%%%%%%%%%%%%%%%%%%%%%%%%%%%%%%%%%%%%%%%%%%%%%%%%%%%%%

Given any subset $\Omega\subset\R^d$, one has a continuous trace operator $T_\Omega:C(\R^d)\rightarrow C(\Omega)$ that acts on functions by restricting them to $\Omega$, i.e. $T_\Omega(f) = f|_\Omega$. Extending the trace operator to other classes of functions is a well-studied subject (see \cite[Chapter VII]{Adams:1975SobolevSpaces} or \cite[Chapter 5]{BesovEtAl1979:IntegralRep}).  In the case when $\Omega$ is a submanifold of $\R^d$, the basic idea is that restricting a function $f\in W^{\tau}_p(\R^d)$ to the submanifold ``costs'' $1/p$ units of smoothness for each dimension, e.g. if $\M$ is a smooth submanifold of dimension $k$, then $f|_\M\in  W^{\tau - (d-k)/p}_p(\M)$. However, there are restrictions. For example, when $p>2$ the imbeddings do not necessarily hold when $\tau - (d-k)/p$ is a nonnegative integer \cite[Theorem 7.58]{Adams:1975SobolevSpaces}. One can get around this and other subtle issues by instead considering Besov spaces. Nevertheless, in the case we will be interested in, i.e. $p=2$, we will not have to shift our focus to these spaces. The following is the trace result we require (see, for example \cite[Section 25]{BesovEtAl1979:IntegralRep}).

\begin{proposition}\label{proposition:trace}
Let $\tau > 0$ and $1\leq k\leq d$. Let $\M$ be a smooth $k$-dimensional compact embedded submanifold of $\R^d$. Then the trace operator $T_\M$ extends to a continuous operator mapping $H^\tau(\R^d)$ onto $H^{\tau - (d-k)/2}(\M)$. Further, there is a reverse imbedding, i.e. there is a bounded linear map $E_{\M}:H^{\tau - (d-k)/2}(\M)\rightarrow H^\tau(\R^d)$ such that $E_{\M}u |_\M = u$ for all $u\in H^{\tau - (d-k)/2}(\M)$.  
\end{proposition}

\noindent We emphasize that the extension $E_\M$ is independent of $\tau$.

%%%%%%%%%%%%%%%%%%%%%%%%%%%%%%%%%%%%%%%%%%%%%%%%%%%%%%%%%%%%%%%%%%%%%%%%%%%
\subsection{Positive Definite Kernels and Native Spaces}
%%%%%%%%%%%%%%%%%%%%%%%%%%%%%%%%%%%%%%%%%%%%%%%%%%%%%%%%%%%%%%%%%%%%%%%%%%%

Let $\Omega\subset\R^d$, and recall that an $N\times N$ matrix $A$ is positive definite if given any nonzero $c\in \R^N$, the quadratic form $c^TAc$ is strictly positive. We say that a kernel $\phi:\Omega\times\Omega\rightarrow\R$ is \emph{positive definite} on $\Omega$ if given any finite set of distinct nodes $\{x_1,x_2,\ldots,x_N \}\subset\Omega$, the $N\times N$ Gram matrix with entries $A_{i,j}=\phi(x_i,x_j)$ is positive definite (and hence invertible). All of the kernels we consider in this paper have this property, so we will take ``kernel'' to mean ``positive definite kernel.'' A kernel that depends only on the distance between its arguments, i.e. $\phi(x,y)=\phi(\|x-y\|)$,  is called a \emph{radial basis function} (RBF). This subclass of kernels is of particular importance; they are essentially one-dimensional.

Typically approximation is well understood for target functions coming from the so-called native space of the kernel, which is a reproducing kernel Hilbert space generated by the kernel. We define the \emph{native space} of a given kernel $\phi$ on $\Omega$ in the usual way, that is by taking the closure of the pre-Hilbert space
\[
F_{\phi}=\left\{f\,\,\left| \,\,f=\sum_{j=1}^N c_j\phi(\cdot\,,x_j),\,\,x_j\in \Omega,\,\,c_j\in \R\right. \right\}
\]
in the inner-product
\[
\left\langle \sum_{j=1}^Nc_j\phi(\cdot\,,x_j),\sum_{k=1}^M d_k\phi(\cdot\,,y_k)\right\rangle_{F_\phi}=\sum_{k=1}^M \sum_{j=1}^N d_k\psi(y_k,x_j)c_j. 
\]
We will denote the native space of $\phi$ by $\mathcal{N}_{\phi}$. 

Defined in this way the native space can seem quite abstract, and characterizing it in terms of more well-known function spaces is helpful in finding classes of functions that can be approximated by shifts of $\phi$. The structure of the native space for kernels of finite smoothness on Euclidean spaces is well-known. To be more specific, if $\phi$ is a kernel on $\R^d$ whose Fourier transform, denoted by $\widehat{\phi}$, has algebraic decay, i.e.
\begin{equation}\label{eq:algdecay}
\widehat{\phi}(\xi)\sim (1+\|\xi\|_2^2)^{-\tau},\,\,\, \tau>d/2,
\end{equation}
then $\mathcal{N}_{\phi} = H^{\tau}(\R^d)$ with equivalent norms. 

We end this section by stating a well-known, indispensable property of kernel interpolants. If $f\in \mathcal{N}_{\phi}$, then $I_{X,\phi}f$ is the orthogonal projection (in $\mathcal{N}_\phi$) of $f$ onto the subspace $\underset{x_j\in X}{\text{span}}\,\phi(\cdot,x_j)$. This immediately gives one
\[\|f - I_{X,\phi}f\|_{\mathcal{N}_{\phi}} \leq \| f \|_{\mathcal{N}_{\phi}} \mbox{\hspace{.2in}and\hspace{.2in}}\|I_{X,\phi}f\|_{\mathcal{N}_{\phi}} \leq \| f \|_{\mathcal{N}_{\phi}}.\]

%%%%%%%%%%%%%%%%%%%%%%%%%%%%%%%%%%%%%%%%%%%%%%%%%%%%%%%%%%%%%%%%%%%%%%%%%%%
\section{Native Spaces for Restricted Kernels}
%%%%%%%%%%%%%%%%%%%%%%%%%%%%%%%%%%%%%%%%%%%%%%%%%%%%%%%%%%%%%%%%%%%%%%%%%%%

As mentioned previously, one approach to finding functions that can be approximated by a given kernel is to determine the native space of that kernel. In this section we characterize the native space for positive definite kernels that have been restricted to an embedded manifold. Let $\phi(x,y)$ be a positive definite kernel on $\R^d$ satisfying (\ref{eq:algdecay}). Given an embedded submanifold $\M\subset \R^d$, we define the kernel $\psi$ on $\M$ by restricting $\phi$ to the manifold, i.e.
\[
\psi(\cdot\,,\cdot):= \phi(\cdot\,,\cdot)|_{\M\times \M}.
\]
As stated in the introduction, it is clear that $\psi$ inherits the positive definiteness of $\phi$. Thus $\psi$ generates a native space on $\M$, which we denote by $\mathcal{N}_{\psi}$. 

In the case of $\M = \Sphere^{d-1}\subset\R^d$, the native spaces of restricted RBFs have been historically studied by investigating the decay of the kernel's Fourier-Legendre coefficients \cite{NarcSunWard2007:SBFRBFconnection,zuCastellFilbir2005:RestrictedRBFs}. However, the methods used in these papers are difficult to apply when the manifold is more general, e.g. the connection between the intrinsic Fourier coefficients and Fourier transform of the extrinsic kernel might be unclear, so we need to take a different perspective. The arguments we use ultimately rely on the variational approach due to Madych and Nelson, which completely avoids the use of the Fourier transform \cite{MadychNelson1988:MultivariateInterpolation}. The following is from Section 8 of that paper. 

\begin{proposition}\label{proposition:nativechar1}
Let $\Omega \subset \R^d$ and let $\kappa:\Omega\times\Omega\rightarrow \R$ be a positive definite kernel. Finally, let $L(\Omega)$ the set of linear functionals given by finite linear combinations of point evaluations, i.e.,
\[L(\Omega):=\left\{\lambda = \sum_{j=1}^N\alpha_j \delta_{x_j}\,\,\left|\,\,N\in\mathbb{N},\,\alpha\in\R^d,\, x_j\in \Omega \right.\right\}.\]
Then we have $f \in \mathcal{N}_{\kappa}$ if and only if there is a constant $C_f$ so that 
\[|\lambda(f)|\leq C_f \|\lambda\|_{\mathcal{N}_{\kappa}^*}.\]
Further, 
\[\|f\|_{\mathcal{N}_{\kappa}}=\sup_{{\lambda\in L(\Omega)},{\lambda\neq 0}}\frac{\lambda(f)}{\|\lambda\|_{\mathcal{N}_{\kappa}^*}}.\]
\end{proposition}

Before we can prove our results, we need a lemma. The following is due to Schaback \cite[Section 9]{Schaback1999:NativeSpacesI}, and can also be found in Wendland's book \cite[Theorems 10.46 and 10.47]{Wendland:2004}. We include the proof here to illustrate the role Proposition \ref{proposition:nativechar1} in our study.

\begin{lemma}\label{lemma:nativetrace}
Let $\phi$ and $\psi$ be related as above. Then we have the following
\begin{enumerate}
\item There is a natural linear operator $E :\mathcal{N}_{\psi}\rightarrow  \mathcal{N}_{\phi}$ such that $Ef|_{\M} = f$ and 
\[\|E f\|_{\mathcal{N}_{\phi}}= \|f\|_{\mathcal{N}_{\psi}}\]
\item The native spaces of $\phi$ and $\psi$ are related via 
\[
\mathcal{N}_{\psi}=T_{\M}\left(\mathcal{N}_{\phi}\right).
\]
\item The trace operator $T_{\M}:\mathcal{N}_{\phi}\rightarrow  \mathcal{N}_{\psi}$ is continuous with $\|T_{\M}\|\leq 1$.
\end{enumerate}
\end{lemma}

\begin{proof}
First we define the extension operator $E: F_{\psi}\rightarrow F_{\phi}$ by
\[E\left(\sum_{j=1}^N c_j\psi(\cdot,x_j)\right) = \sum_{j=1}^N c_j\phi(\cdot,x_j).\]
Clearly we have that $\|E f\|_{\mathcal{N}_{\phi}}= \|f\|_{\mathcal{N}_{\psi}}$ for all $f\in F_{\psi}$. Now using a density argument we can extend $E$ to map $\mathcal{N}_{\psi}$ to $\mathcal{N}_{\phi}$. Since $E$ preserves norms for the dense subsets, we conclude that $E$ is an isometry.

For 2, it is clear from 1 that $\mathcal{N}_{\psi}\subseteq T_{\M}\left(\mathcal{N}_{\phi}\right)$. Focusing now on the reverse inclusion, let $f\in T_{\M}\left(\mathcal{N}_{\phi}\right)$. Then there is a $g\in\mathcal{N}_{\phi}$ such that $T_{\M}g=f$. By Proposition \ref{proposition:nativechar1}, we need to find a constant $C_f$ such that
\[|\lambda(f)|\leq C_f \|\lambda\|_{\mathcal{N}_{\psi}^*}\]
for all $\lambda \in L(\M)$. First, we have $\lambda(f) = \lambda(g)$ and $\|\lambda\|_{\mathcal{N}_{\psi}^*}=\|\lambda\|_{\mathcal{N}_{\phi}^*}$ for all $\lambda\in L(\M)$, giving
\[|\lambda(f)|=|\lambda(g)|\leq C_g \|\lambda\|_{\mathcal{N}_{\phi}^*}=C_g \|\lambda\|_{\mathcal{N}_{\psi}^*}.\]
Thus $f\in\mathcal{N}_{\psi}$. Further, note that since $L(\M)\subset L(\R^d)$ we have
\begin{eqnarray*}
\|f\|_{\mathcal{N}_{\psi}} & = &\sup_{{\lambda\in L(\M)},{\lambda\neq 0}}\frac{\lambda(f)}{\|\lambda\|_{\mathcal{N}_{\psi}^*}} = \sup_{{\lambda\in L(\M)},{\lambda\neq 0}}\frac{\lambda(g)}{\|\lambda\|_{\mathcal{N}_{\phi}^*}} \\
&\leq &\sup_{{\lambda\in L(\R^d)},{\lambda\neq 0}}\frac{\lambda(g)}{\|\lambda\|_{\mathcal{N}_{\phi}^*}}=\|g\|_{\mathcal{N}_{\phi}}.
\end{eqnarray*}
\noindent This shows that the trace operator is continuous with norm less than one, and this completes the proof.
\end{proof}

Now we are ready to present a characterization for the native space of the restricted kernel.

\begin{theorem}\label{theorem:restrictednative}
If $\phi$ satisfies (\ref{eq:algdecay}), then $\mathcal{N}_{\psi}=H^{\tau-(d-k)/2}(\M)$ with equivalent norms.
\end{theorem}

\begin{proof}
Our choice of $\phi$ gives $\mathcal{N}_\phi=H^\tau(\R^d)$ with equivalent norms.  By Proposition \ref{proposition:trace} and Lemma \ref{lemma:nativetrace} we have
\[\mathcal{N}_{\psi} = T_\M(\mathcal{N}_\phi) =  T_\M(H^\tau(\R^d)) = H^{\tau-(d-k)/2}(\M).\]
Next we show that the native space norm dominates the Sobolev norm by a constant factor. Note that $f = T_\M Ef$ and that the trace operator is continuous on the appropriate Sobolev spaces. Now we have
\begin{eqnarray*}
\|f\|_{H^{\tau-(d-k)/2}(\M)} &= & \|T_\M Ef\|_{H^{\tau-(d-k)/2}(\M)} \leq C \|Ef\|_{H^{\tau}(\R^d)} \\
&\leq &C\|Ef\|_{\mathcal{N}_{\phi}} = C\|f\|_{\mathcal{N}_{\psi}}.
\end{eqnarray*}
It is a well-known consequence of the Interior Mapping Theorem that in such situations the norms must be equivalent.
\end{proof}

%%%%%%%%%%%%%%%%%%%%%%%%%%%%%%%%%%%%%%%%%%%%%%%%%%%%%%%%%%%%%%%%%%%%%%%%%%%
\section{Interpolation via Restricted Kernels}
%%%%%%%%%%%%%%%%%%%%%%%%%%%%%%%%%%%%%%%%%%%%%%%%%%%%%%%%%%%%%%%%%%%%%%%%%%%

We are ready to shift our focus to the approximation of functions within certain Sobolev spaces using restricted kernels. The methods we use are related to those used by Narcowich \emph{et. al.} \cite{NarcWardWend:2006bernstein}, although the present situation is different enough so that there is still work to be done before we can successfully apply their methods. 

%%%%%%%%%%%%%%%%%%%%%%%%%%%%%%%%%%%%%%%%%%%%%%%%%%%%%%%%%%%%%%%%%%%%%%%%%%%
\subsection{Node Measures on Submanifolds of $\R^d$}
%%%%%%%%%%%%%%%%%%%%%%%%%%%%%%%%%%%%%%%%%%%%%%%%%%%%%%%%%%%%%%%%%%%%%%%%%%%

Given a finite node set $X$ from a metric space $\Theta$, error estimates are typically given in terms of the \emph{fill distance}, or \emph{mesh norm} of the points, which is defined to be
\[h_{X,\Theta} := \sup_{x\in \Theta}\min_{x_j\in X}d_{\Theta}(x,x_j),\]
where $d_{\Theta}$ is the distance metric between $x$ and $y$ intrinsic to $\Theta$. Another important measure is the \emph{separation radius}, given by
\[q_{X,\Theta}:= \min_{\overset{x_j,x_k\in X}{x_j\neq x_k}}\frac{1}{2}d_{\Theta}(x_j,x_k).\]
Since we wish to study approximation on a manifold, it is important that our results be stated in terms of the mesh norm and separation radius intrinsic to the manifold. Note that our node sets simultaneously reside in several different metric spaces, namely $\R^d$, $\M$, and $\R^k$ (through chart mappings). We will need the ``node measures'' on all three of these spaces, and it will be convenient to put them on equal footing. This is not a difficult task, but it is a detail that must be dealt with nonetheless. 

First we show that $q_{X,\M}\sim q_{X,\R^d}$.
\begin{theorem}\label{theorem:equivalent separation radius}
Let $\M\subset \R^d$ be a smooth compact embedded submanifold of dimension $k<d$. Then $d_\M(x,y)\sim \|x - y\|$ for all $x,y\in \M$. In particular, we have $q_{X,\M}\sim q_{X,\R^d}$ for all finite node sets $X\subset \M$.
\end{theorem}

\begin{proof}
Clearly we have $\|x - y\|\leq d_\M(x,y)$. Now we find a bound in the other direction. Since $\M$ is smooth and compact, we know from Proposition \ref{proposition:tubularngd} that there is a compact tubular neighborhood $U_\epsilon(\M)$ with normal radius $\epsilon$ and smooth retraction $\mathfrak{R}:U_\epsilon(\M)\rightarrow \M$. Given $x,y \,\,\,\in \M$, we consider two cases: $\|x -y\|\geq \epsilon$ and $\|x -y\|< \epsilon$. If $\| x - y\| \geq \epsilon$, we have
\[
d_{\M}(x,y) = \frac{d_\M(x,y)}{\|x - y\|}\|x - y\| \leq \frac{\mbox{diam}(\M)}{\epsilon}\|x - y\|.
\]

Now assume $0< \|x - y\| < \epsilon$. Notice that $y\in B(x,\epsilon)\subset U_\epsilon(\M)$, so the parameterized line $l:[0,1]\rightarrow \R^d$ starting at $x$ and ending at $y$ is completely contained in $U_\epsilon(\M)$. Thus $\mathfrak{R}\circ l$ is a smooth parameterized curve in $\M$ starting at $x$ and ending at $y$. Now we use the arclength definition of the distance metric and the fact that $l'(t)=y - x$ to get
\begin{eqnarray*}
d_\M(x,y) & = & \inf_{\underset{\underset{\gamma(1)=y}{\gamma(0)=x}}{\gamma:[0,1]\rightarrow \M}}\int_{0}^1\left\|\frac{d\gamma}{dt}\right\|\,dt \leq \int_{0}^1\left\|\frac{d}{dt}(\mathfrak{R}\circ l)\right\|\,dt \\
& = & \int_{0}^1\| (D\mathfrak{R})|_{l(t)}l'(t)\| \,dt \leq \int_{0}^1\|D\mathfrak{R}\|_{\infty} \|x - y\| \,dt \\
& = & \|D\mathfrak{R}\|_{\infty} \|x - y\|,
\end{eqnarray*}
where $\|D\mathfrak{R}\|_{\infty}$ denotes a bound on the matrix norm of the Jacobian matrix $D\mathfrak{R}$, which we know to be bounded since $\mathfrak{R}$ is smooth and its domain is compact. The result now follows.
\end{proof}
\noindent The reader should note that the discrepancy between $q_{X,\M}$ and $q_{X,\R^d}$ could be quite large, e.g. consider nodes taken at the poles of a flattened ball. However, if the node set is sufficiently dense this will not be an issue. 

Now we shift our focus to the mesh norm. The first step is to show that distances are preserved under chart mappings. This has been taken care of in the case of homogeneous manifolds in \cite[Proposition 3.3]{LevesleyRagozin:2007RBFsHomogenousManifolds}, and we cite its proof to deal with our situation. 

\begin{proposition}\label{prop:equivalent local distance}
Let $\M$ be a compact smooth manifold of dimension $k$. Then there exists an atlas $\calA = \{(\Psi_j,U_j)\}$ for $\M$ and associated positive constants $c_1$, $c_2$ such that for all $x,y\in \M$, if $x,y\in U_j$ for some $j$ we have
\[c_1\|\Psi_j(x) - \Psi_j(y)\|\leq d_\M(x,y)\leq c_2\|\Psi_j(x) - \Psi_j(y)\|.\]
\end{proposition}

\begin{proof}
Since $\M$ is a smooth and compact manifold, we get a smooth atlas $\calA = \{(\Psi_j,U_j)\}$ containing finitely many charts for $\M$. By refining our charts as necessary, we may assume that the images of these charts is equal to a ball in $\R^k$ centered around the origin. Since our manifold is smooth and the images of our charts are convex, the arguments in the proof of \cite[Proposition 3.3]{LevesleyRagozin:2007RBFsHomogenousManifolds} follow to find positive constants $c_{1,j}$ and $c_{2,j}$ such that
\[c_{1,j}\|\Psi_j(x) - \Psi_j(y)\|\leq d_\M(x,y)\leq c_{2,j}\|\Psi_j(x) - \Psi_j(y)\|\mbox{   }\forall\,\,x,y\in U_j.\]
Now simply set $c_1:= \min_{j} c_{1,j}$ and $c_2:= \max_{j} c_{2,j}$, and the proof is complete.
\end{proof}

From here on out, we let $\calA = \{(\Psi_j,U_j)\}$ be such an atlas for $\M$, with each $\Psi_j$ being a bijection from $U_j$ to some ball $B'(0,r_j)\subset\R^k$. In the next section we will be able to easily obtain error bounds on each chart in terms of the mesh norm in $\R^k$, $h_{\Psi_j(X),B'(0,r_j)}$. In an effort to use a mesh norm intrinsic to the manifold, consider the following. With the result above, we have:
\begin{equation*}
h_{X\cap U_j,U_j} \sim h_{\Psi_j(X),B'(0,r_j)},
\end{equation*}
which gives us in particular that
\begin{equation}\label{eq: chart mesh bound}
h_{\Psi_j(X),B'(0,r_j)} \leq C h_{X\cap U_j,U_j}.
\end{equation}
Using this, a useful global density of $X$ on $\M$ is $h^*:=\max_j  h_{X\cap U_j,U_j}$, but this is not as strong as the mesh norm $h_{X,\M}$. Indeed, it is not hard to show that $h_{X,\M}\leq h^*$. However, for node sets yielding a mesh norm $h_{X,\M}$ small enough, $h^*$ and $h_{X,\M}$ are equivalent.

\begin{theorem}\label{theroem: equivalent fill distance}
Let $\M$ be a compact smooth manifold of dimension $k$, and let $\calA = \{(\Psi_j,U_j)\}$ be an atlas for $\M$ satisfying Proposition \ref{prop:equivalent local distance}. Then there exists constants $h_0, C>0$ so that for any finite node set $X\subset \M$ with $h_{X,\M}< h_0$, we have 
\[h_{X\cap U_j,U_j}\leq C\, h_{X,\M}\mbox{  for all }U_j.\]
\end{theorem}

\begin{proof}
Recall that we may assume that every map $\Psi_j$ is a smooth bijection onto a ball of radius $r_j>0$. Define $r = \min_j r_j$. We will exploit geometry of the situation by using the fact that balls satisfy an interior cone condition. Given a point $x\in \R^k$, unit vector $\xi$, radius $R$, and angle $\theta\in (0,\pi/2)$, we define the cone
\[C(x,\xi,R,\theta):=\{ x + \lambda y\,\,\,|\,\,\,y\in \R^k,\,\|y\|=1,\,y^T\xi\geq \cos(\theta),\,\lambda\in[0,R]\}.\]
Below are the geometric facts we require \cite[Lemma 3.7, Lemma 3.10]{Wendland:2004}:
\begin{enumerate}
\item Let $C(x,\xi, R,\theta)$ be a cone of radius $R$ and angle $\theta$. If $h<R/(1+\sin(\theta))$, then the cone contains a closed ball of radius $h\sin(\theta)$ centered a distance $h$ away from $x$. 
\item Every ball with radius $R>0$ satisfies an interior cone condition with radius $R$ and angle $\theta = \pi/3$, i.e. for every $x$ in the ball, a unit vector $\xi(x)$ exists such that $C(x,\xi(x),R,\theta)$ is contained in the ball.
\end{enumerate}
With this insight, we choose $h_0:= r c_1\sin(\theta)/(1 + \sin(\theta))$, where $\theta = \pi/3$ and  $c_1$ is the constant from the lower bound in Proposition \ref{prop:equivalent local distance}. Assuming $h_{X,\M}\leq h_0$, we will show that given $x\in U_j$, there is a point within $X\cap U_j$ whose distance from $x$ is comparable to $h_{X,\M}$. 

Let $x\in U_j$ and consider the ball $\Psi_j(U_j)$. Since the ball satisfies an interior cone condition with radius $r_j\geq r$, by the remarks above we can find a closed ball of radius $h_{X,\M}/c_1$ centered at some $\Psi_j(y)$ satisfying $\|\Psi_j(x) - \Psi_j(y)\| = h_{X,\M}/(c_1\sin(\theta))$. We denote this ball by $B_1$, and note that we have $\Psi_j^{-1}(B_1)\subset U_j$. Further, since $\Psi_j$ and its inverse are open topologically, $\Psi_j^{-1}(B_1)$ is a neighborhood of $y$ and there is some closed ball centered at $y$ completely contained in $\Psi_j^{-1}(B_1)$. 

Let $\rho$ be the maximum radius of such a ball, i.e. $B(y,\rho)\subset \Psi_j^{-1}(B_1)$ and given any $\rho'>\rho$ there is a $z'\in B(y,\rho')$ such that $z'\notin \Psi_j^{-1}(B_1)$. Note that this maximum radius must be attained at some $z \in \partial{\Psi_j^{-1}(B_1)}$. Also note that $z$ must map to the boundary of $B_1$ through $\Psi_j$: if not, then because $\Psi_j^{-1}$ is open, we could find a neighborhood of $z$ completely contained in $\Psi_j^{-1}(B_1)$, contradicting the fact that it was chosen on the boundary. Thus we have:

\[\rho =  d_{\M}(y,z) \geq c_1 \|\Psi_j(y) - \Psi_j(z)\| = c_1 \left(\frac{h_{X,\M}}{c_1}\right)= h_{X,\M}.\]

Since $\rho\geq h_{X,\M}$, it follows that there must be a point $x_k\in X\cap B(y,\rho)\subset \Psi_j^{-1}(B_1) \subset U_j$.
Now we have:
\begin{eqnarray*}
d_\M(x,x_k) &\leq &d_\M(x,y) + d_\M(y,x_k) \\
& \leq & c_2\|\Psi_j(x) - \Psi_j(y)\| + c_2 \|\Psi_j(x_k) - \Psi_j(y)\|\\
& \leq & c_2 \frac{h_{X,\M}}{c_1\sin(\theta)} + c_2 \frac{h_{X,\M}}{c_1} \leq \frac{c_2}{c_1}\left(\frac{2 + \sqrt{3}}{\sqrt{3}}\right) h_{X,\M} = C h_{X,\M}.
\end{eqnarray*}
Since $x$ and $U_j$ were arbitrary, this proves that $h_{X\cap U_j,U_j}\leq C h_{X,\M}$.
\end{proof}

%%%%%%%%%%%%%%%%%%%%%%%%%%%%%%%%%%%%%%%%%%%%%%%%%%%%%%%%%%%%%%
\subsection{Interpolation Error Estimates}
%%%%%%%%%%%%%%%%%%%%%%%%%%%%%%%%%%%%%%%%%%%%%%%%%%%%%%%%%%%%%%

We now have the tools necessary to provide error bounds. First we present estimates that apply to target functions coming from the native space of the approximating kernel. After this, we will give estimates for target functions that are not smooth enough to be within the native space.

To derive our estimates, we will first make use of the ``many zeros'' Sobolev sampling inequality of Narcowich, Ward and Wendland,  which allows one to extract the appropriate powers of the mesh norm from the error function \cite{NarcWardWendland:2005ScatteredZeros}. Here is a statement of that result, stated in slightly stronger form as in  \cite{NarcWardWend:2006bernstein}.

\begin{proposition}\label{prop:sobolev_zero_est}
Let $\Omega$ be a compact subset of $\R^n$ satisfying an interior cone condition. Let $s>0$, $1\leq p<\infty$, $1\leq q\leq\infty$, and let $\mu$ be an integer satisfying $\lfloor s \rfloor >\mu +n/p$, or $p=1$ and $\lfloor s\rfloor \geq \mu +n$. Also, let $X\subset\Omega$ be a discrete set with mesh norm $h_{X,\Omega}$.  Then there is a constant depending only on $\Omega$ such that if $h_{X,\Omega}\leq C_{\Omega}$ and if $u\in W^{s}_{p}(\Omega)$ satisfies $u|_{X}=0$, then
\begin{equation}
|u|_{W_{q}^{\mu}(\Omega)}\leq Ch_{X,\Omega}^{s-\mu-n(1/p-1/q)_{+}}|u|_{W^{s}_{p}(\Omega)},
\end{equation}
where $(x)_{+}=x$ if $x\geq 0$ and is $0$ otherwise. Here the constant $C$ is independent of $h_{X,\Omega}$ and $u$.
\end{proposition}

Recall our notation for kernel interpolants: given a finite node set $X$, kernel $\psi$ and target function $f$, we let $I_{X,\psi}f$ denote the interpolant to $f$ on $X$ found via shifts of $\psi$. Now we are ready to state our first approximation result. 

\begin{theorem}\label{theorem: native error estimate}
Let $\M$ be a $k$-dimensional submanifold of $\R^d$, $\phi$ be a positive definite kernel satisfying (\ref{eq:algdecay}), and define $\psi$ by restricting $\phi$ to $\M$. Let $s = \tau - (d - k)/2$, and let $\mu$, and $q$ be as in Proposition \ref{prop:sobolev_zero_est} with $n=k$ and $p=2$. Then there is a constant $h_\M$ such that if a finite node set $X\subset \M$ satisfies $h_{X,\M}\leq h_\M$, then for all $f\in H^{s}(\M)$ we have
\[
\|f - I_X f\|_{W^{\mu}_{q}(\M)}\leq C h_{X,\M}^{s-\mu-k(1/2-1/q)_{+}}\|f\|_{H^{s}(\M)}.
\]
\end{theorem}

\begin{proof}
Let $\calA = \{(U_j,\Psi_j)\}$ be an atlas for $\M$ satisfying Proposition \ref{prop:equivalent local distance}. We will choose $h_\M$ small enough so that Theorem \ref{theroem: equivalent fill distance} holds, and so that Proposition \ref{prop:sobolev_zero_est} can be applied to the images of all patches in $\R^k$. The norm of the error is given by
\begin{eqnarray*}
\|f - I_X f\|_{W^{\mu}_{q}(\M)} & = & \left(\sum_{j=1}^N\|\pi_j(f - I_X f)\|^2_{W^{\mu}_q(\R^k)}\right)^{1/2}. 
\end{eqnarray*}
Any function projected under $\pi_j$ is supported on $\Psi_j(U_j)$, so we have
\[\|\pi_j(f - I_X f)\|^2_{W^{\mu}_q(\R^k)} = \|\pi_j(f - I_X f)\|^2_{W^{\mu}_q(\Psi_j(U_j))}.\]
Note that $\pi_j(f - I_X f)$ is a Sobolev function with many zeros on the set $\Psi_j(X\cap U_j)$.  Applying Proposition \ref{prop:sobolev_zero_est} and using (\ref{eq: chart mesh bound}) with Theorem \ref{theroem: equivalent fill distance} gives us
\begin{eqnarray*}
\|\pi_j(f - I_X f)\|_{W^{\mu}_q(\Psi_j(U_j))}\leq C h_{X,\M}^{s - \mu - k(1/2 - 1/q)_+}\|\pi_j(f - I_X f)\|_{H^{s}(\Psi_j(U_j))}, 
\end{eqnarray*}
where the constant is independent of $X$ and $f$. Applying this estimate to all patches gives us
\[\|f - I_X f\|_{W^{\mu}_{q}(\M)}\leq C h_{X,\M}^{s-\mu-k(1/2-1/q)_{+}}\|f - I_X f\|_{H^{s}(\M)}.\]

Now recall that since we chose $\phi$ to satisfy (\ref{eq:algdecay}), the native space of $\psi$ is equal to $H^{s}(\M)$ with equivalent norms. Now use this and the fact that the kernel interpolants have a best approximation property to get
\begin{equation*}
\|f - I_X f\|_{H^{s}(\M)} \leq C \|f - I_X f\|_{\mathcal{N}_{\psi}} \leq C \|f\|_{\mathcal{N}_{\psi}} \leq C \|f\|_{H^{s}(\M)}\nonumber.
\end{equation*}
This completes the proof.
\end{proof}

%%%%%%%%%%%%%%%%%%%%%%%%%%%%%%%%%%%%%%%%%%%%%%%%%%%%%%%%%%%%%%%%%%%%%%%%%
%%%%%%%%%%%%%%%%%%%%%%%%%%%%%%%%%%%%%%%%%%%%%%%%%%%%%%%%%%%%%%%%%%%%%%%%%
%%%%%%%%%%%%%%%%%%%%%%%%%%%%%%%%%%%%%%%%%%%%%%%%%%%%%%%%%%%%%%%%%%%%%%%%%

When the target function is very smooth, there is a ``doubling trick'' from spline theory that can be used to increase the order of $h_{X,\M}$ in the error estimates. This was first commented on for RBFs in $\R^d$ by Schaback in \cite{SCHA:1999} and has also been observed in other RBF-related contexts~\cite{FuselierWright:2009VectorDecomposition,MORTON_NEAMTU:2002}. The error doubling on more general domains, including Riemannian manifolds, was considered in \cite{Schaback2000:NativeSpacesII}, and Theorem 5.1 of that paper gives pointwise error estimates for very smooth functions in terms of the so-called \emph{Power Function}. Using similar methods along with Theorem \ref{theorem: native error estimate}, we get Sobolev error estimates doubling the order of $h_{X,\M}$.

A main ingredient for the doubling trick is to rewrite the native space inner product in terms of an $L_2$ inner product of functions that have been acted upon by a pseudodifferential operator depending on the kernel.  This is achieved in $\R^d$ using the Fourier transform, and in the case of the sphere by using eigenfunction expansions of the Laplace-Beltrami operator $\Delta$. For more general domains, one can use a well-known result of Mercer to find the appropriate native space machinery. The tools outlined below can be found in \cite{Schaback2000:NativeSpacesII}.

We define the integral operator $T:L_2(\M)\rightarrow L_2(\M)$ by
\[Tf(x) :=\int_{\M}\psi(x,y)f(y)\,dy\]
Since our kernel is continuous and $\M$ is compact, we may invoke Mercer's theorem (see \cite[Theorem 1.1]{Ferreira2009}, for example). Mercer's theorem guarantees a countable set of positive eigenvalues $\lambda_1\geq \lambda_2\geq \cdots > 0$ and continuous eigenfunctions $\{\varphi_n\}_{n\in \mathbb{N}}$ such that $T\varphi_n=\lambda_n\varphi_n$. Further, $\{\varphi_n\}_{n\in \mathbb{N}}$ provides an orthonormal basis for $L_2(\M)$, and $\psi(x,y)$ has the expansion
\[\psi(x,y) = \sum_{n=1}^{\infty}\lambda_n\varphi_n(x)\varphi_n(y).\]
Lastly, with these tools one has the following characterization of the native space. A proof can be found in \cite[Sections 7,8]{Schaback2000:NativeSpacesII}.
\begin{proposition}
Let $\psi$ be a positive definte kernel on $\M$. Then its native space is given by
\[\mathcal{N}_{\psi}=\left\{f\in L_2(\M)\,\left|\, \sum_{n=1}^{\infty}\frac{1}{\lambda_n} (f,\varphi_n)^2_{L_2(\M)}\right. \right\}.\] 
Also, for $f,g\in \mathcal{N}_{\psi}$ the inner product has the representation
\[\left\langle f,g\right\rangle_{\mathcal{N}_\psi}=\sum_{n=1}^{\infty} \frac{1}{\lambda_n} (f,\varphi_n)_{L_2(\M)}(g,\varphi_n)_{L_2(\M)}.\]
\end{proposition}

Along with the integral operator $T$ comes pseudodifferential operators $T^{-r}$, $r>0$, defined formally by
\[T^{-r}f(x) := \sum_{n=1}^{\infty}\frac{1}{\lambda_n^r} (f,\varphi_n)_{L_2(\M)}\varphi_n(x).\]
The above proposition tells us that a function $f$ resides in the native space if and only if $T^{-1/2}f\in L_2(\M)$. Thus we expect functions such that $T^{-1}f\in L_{2}(\M)$ to be at least twice as smooth. Below we show that these smoother functions enjoy faster convergence rates.

\begin{corollary}\label{theorem: doubling trick}
Let $\psi$ and $s$ be as in Theorem \ref{theorem: native error estimate}, and let $f\in \mathcal{N}_{\psi}$ be such that $T^{-1}f\in L_2(\M)$. Then we have
\[\|f - I_X f\|_{L_2(\M)}\leq C h_{X,\M}^{2s} \|T^{-1}f\|_{L_2(\M)}.\]
\end{corollary}

\begin{proof}
First, the arguments in Theorem \ref{theorem: native error estimate} give us that
\begin{equation}\label{eq: double zeros}
\|f - I_X f\|^2_{L_2(\M)}\leq C h_{X,\M}^{2s}\|f - I_X f\|_{\mathcal{N}_{\psi}}^2.
\end{equation}
Recall that the error $g:=f - I_X f$ and $I_Xf$ are orthogonal in $\mathcal{N}_{\psi}$. Using this, the above proposition, and a Cauchy-Schwartz inequality gives us
\begin{eqnarray*}
\|f - I_X f\|^2_{\mathcal{N}_{\psi}} & = & \left\langle g, f\right\rangle_{\mathcal{N}_\psi} = \sum_{n=1}^{\infty} \frac{1}{\lambda_n} (f,\varphi_n)_{L_2(\M)}(g,\varphi_n)_{L_2(\M)} \\
& \leq & \left(\sum_{n=1}^{\infty} \frac{1}{\lambda_n^2} (f,\varphi_n)_{L_2(\M)}^2\right)^{1/2}\left(\sum_{n=1}^{\infty}(g,\varphi_n)_{L_2(\M)}^2\right)^{1/2} \\
& = & \|T^{-1}f\|_{L_2(\M)}\|g\|_{L_2(\M)}.
\end{eqnarray*}
This along with (\ref{eq: double zeros}) finishes the proof.
\end{proof}

%%%%%%%%%%%%%%%%%%%%%%%%%%%%%%%%%%%%%%%%%%%%%%%%%%%%%%%%%%%%%%%%%%%%%%%%%
%%%%%%%%%%%%%%%%%%%%%%%%%%%%%%%%%%%%%%%%%%%%%%%%%%%%%%%%%%%%%%%%%%%%%%%%%
%%%%%%%%%%%%%%%%%%%%%%%%%%%%%%%%%%%%%%%%%%%%%%%%%%%%%%%%%%%%%%%%%%%%%%%%%

Now we shift our attention to functions less smooth than those in the native space. Finding error estimates for functions outside the native space, sometimes called ``escaping'' the native space, has only recently been made possible through the use of ``band-limited'' functions, which are functions with compactly supported Fourier transforms. Estimates for the escape when using radial basis functions restricted to the sphere were given in \cite{NarcSunWard:2007direct}, where the authors used tools intrinsic to the sphere. To deal with more general manifolds, we will lift the problem from the manifold to $\R^d$.

The results that lead to the escape can be quite deep; to be brief we merely list the properties of the band-limited functions needed for our proof. For the interested reader, more details and a complete reference list can be found in the survey paper \cite{Narc2005:RecentErrorSurvey}. The following are consequences of Theorem 3.4 in \cite{NarcWardWend:2006bernstein} and the remarks thereafter.

\begin{proposition}\label{prop: band limited interpolant}
Let $\nu > d/2$, and let $X$ be a finite subset of $\R^d$. If $f\in H^{\nu}(\R^d)$ then there is an $f_\sigma:\R^d\rightarrow \R$ such that
\begin{enumerate}
\item $f_{\sigma}|_X = f|_X$,
\item The Fourier transform of $f_\sigma$ is supported within $B(0,\sigma)$ with $\sigma = C/q_{X,\R^d}$,
\item $\|f - f_\sigma\|_{H^{\nu}(\R^d)}\leq C \|f\|_{H^{\nu}(\R^d)}$,
\item $\|f_{\sigma}\| \leq C \|f\|_{H^{\nu}(\R^d)}$,
\end{enumerate}
where each $C$ represents a constant independent of $X$ and $f$. 
\end{proposition}

As the reader will see below, there is apparently a penalty paid for approximating a rough function, given by the introduction of so-called \emph{mesh ratio}, which we denote by $\rho_{X,\M} := h_{X,\M}/q_{X,\M}$.

\begin{theorem}\label{theorem: escape error estimate}
Let $\M$, $\psi$, $s$ and $q$ be as above. Let $\beta$ be such that $s > \beta > k/2$, and suppose $\mu$ is an integer satisfying $\lfloor \beta \rfloor > \mu + k/2$. Then there is a constant $h_\M$ such that if a finite node set $X\subset \M$ satisfies $h_{X,\M}\leq h_\M$, then for all $f\in H^{\beta}(\M)$ we have
\[
\|f - I_{X,\psi} f\|_{W^{\mu}_{q}(\M)}\leq C h_{X,\M}^{\beta-\mu-k(1/2-1/q)_{+}}\rho_{X,\M}^{s - \beta}\|f\|_{H^{\beta}(\M)}.
\]
\end{theorem}

\begin{proof}
We choose $h_\M$ as before. If $f \in  H^{\beta}(\M)$, the Trace theorem lets us continuously extend $f$ to $H^{\nu}(\R^d)$ via the map $E_{\M}$, where $\nu = \beta + (d - k)/2$. Since $\beta > k/2$ we have $\nu > d/2$, thus allowing us to find a band-limited interpolant $f_{\sigma}$ to $E_{\M}f$ with the approximation properties listed in Proposition \ref{prop: band limited interpolant}.  

To get the appropriate orders of the mesh norm, we again use the many zeros result as in the proof of Theorem \ref{theorem: native error estimate} to get
\[\|f - I_{X,\psi} f\|_{W^{\mu}_{q}(\M)}\leq C h_{X,\M}^{\beta-\mu-k(1/2-1/q)_{+}}\|f - I_{X,\psi}f\|_{H^{\beta}(\M)}.\]
The rest of the proof will follow after bounding $\|f - I_{X,\psi}f\|_{H^{\beta}(\M)}$. Note that since $T_{\M}f_{\sigma}|_X = f|_X$, we have
\[I_{X,\psi}f = I_{X,\psi}T_{\M}f_{\sigma} \]
Using this and the fact that all functions involved can be considered traces of function on $\R^d$ gives us
\begin{eqnarray*}
\|f - I_{X,\psi}f\|_{H^{\beta}(\M)} & \leq  &\|f - T_{\M}f_{\sigma}\|_{H^{\beta}(\M)} + \|T_{\M}f_{\sigma} - I_{X,\psi}f\|_{H^{\beta}(\M)} \\ 
& =  &\|T_{\M}E_{\M}f - T_{\M}f_{\sigma}\|_{H^{\beta}(\M)} + \|T_{\M}f_{\sigma} - I_{X,\psi}T_{\M}f_{\sigma}\|_{H^{\beta}(\M)} \\ 
 & \leq &C \|E_{\M}f - f_{\sigma}\|_{H^{\nu}(\R^d)} + \|T_{\M}f_{\sigma} - I_{X,\psi}T_{\M}f_{\sigma}\|_{H^{\beta}(\M)}.
\end{eqnarray*}
First we concentrate on the leftmost term on the right-hand side. We can use Proposition \ref{prop: band limited interpolant} and the fact that $f$ was continuously extended to get 
\[\|E_{\M}f - f_{\sigma}\|_{H^{\nu}(\R^d)} \leq C \|E_{\M}f\|_{H^{\nu}(\R^d)} \leq C \|f\|_{H^{\beta}(\M)}. \]
With bounding the other term in mind, note that since $f_{\sigma}$ is bandlimited, we have $f_{\sigma}\in \mathcal{N}_{\phi}$, so its restriction to $\M$ is in $\mathcal{N}_{\psi} = H^{s}(\M)$. Thus we can apply Theorem \ref{theorem: escape error estimate} to get
\[\|T_{\M}f_{\sigma} - I_{X,\psi}T_{\M}f_{\sigma}\|_{H^{\beta}(\M)} \leq C  h_{X,\M}^{s-\beta}\|T_{\M}f_{\sigma}\|_{H^{s}(\M)}.\]
The Trace operator is continuous, so we have
\[\|T_{\M}f_{\sigma}\|_{H^{s}(\M)} \leq C \|f_{\sigma}\|_{H^{\tau}(\R^d)},\]
and the fact that $f_\sigma$ is band-limited with bandwidth $\sigma \sim 1/q_{X,\R^d}$ allows us to apply a Bernstein inequality, giving us
\begin{eqnarray*}
\|T_{\M}f_{\sigma}\|_{H^{s}(\M)} & \leq & C \|f_{\sigma}\|_{H^{\tau}(\R^d)} \leq C q_{X,\R^d}^{\nu - \tau} \|f_{\sigma}\|_{H^{\nu}(\R^d)} \\
& \leq & C q_{X,\M}^{\nu - \tau} \|f_{\sigma}\|_{H^{\nu}(\R^d)} =  C q_{X,\M}^{\beta - s} \|f_{\sigma}\|_{H^{\nu}(\R^d)},
\end{eqnarray*}
where in the second to last inequality we have invoked Theorem \ref{theorem:equivalent separation radius}. Continuing with the estimate, Proposition \ref{prop: band limited interpolant} and the fact that $f$ was continuously extended from $\M$ to $\R^d$ gives us
\[ \|f_{\sigma}\|_{H^{\nu}(\R^d)} \leq C  \|E_{\M}f\|_{H^{\nu}(\R^d)} \leq  C  \|f\|_{H^{\beta}(\M)}.\] 
Stringing these inequalities together, we obtain
\begin{eqnarray*}
\|f - I_{X,\psi}f\|_{H^{\beta}(\M)} & \leq & C \|f\|_{H^{\beta}(\M)} +  C  h_{X,\M}^{s-\beta}\|T_{\M}f_{\sigma}\|_{H^{s}(\M)} \\
& \leq & C \|f\|_{H^{\beta}(\M)} +  C  h_{X,\M}^{s-\beta} q_{X,\M}^{\beta - s}\|f\|_{H^{\beta}(\M)} \\
& \leq & C  \rho_{X,\M}^{s - \beta}\|f\|_{H^{\beta}(\M)}.
\end{eqnarray*}
This completes the proof.
\end{proof}

\noindent If the nodes are chosen in a non-uniform way, we see that the error bound above might be quite large. However, if the node sets one is dealing with are more or less uniform, $\rho_{X,\M}$ can be bounded by a constant, and one would obtain the typical approximation rates for target functions of a certain smoothness.

%%%%%%%%%%%%%%%%%%%%%%%%%%%%%%%%%%%%%%%%%%%%%%%%%%%%%%%%%%%%%%%%%%%%%%%%%%%%%%
\section{Numerical results}
%%%%%%%%%%%%%%%%%%%%%%%%%%%%%%%%%%%%%%%%%%%%%%%%%%%%%%%%%%%%%%%%%%%%%%%%%%%%%
We provide numerical results verifying Theorems \ref{theorem: native error estimate} and \ref{theorem: escape error estimate} for target functions inside and outside of the native space, respectively. Two different compact embedded smooth submanifolds in $\R^3$ are considered. The first is a one-dimensional submanifold with the parametric representation
\begin{equation}
\M_1 = \left\{(u,v,w) \in \R^3 \;\left| \;u = \lp 1 + \frac13\cos 6\theta\rp\cos\theta,\; v = \lp 1 + \frac13\cos 6\theta\rp\sin\theta,\; w = \frac13\sin 2\theta,\; 0\leq \theta < 2\pi \right. \right\}.
\label{eq:Curve}
\end{equation}
This curve is displayed in Figure \ref{fig:Curve}.  The second submanifold is a two-dimensional torus with parametric representation
\begin{equation}
\M_2 = \left\{(u,v,w) \in \R^3 \;\left| \;u = \lp 1 + \frac13\cos\lambda\rp\cos\theta,\; v = \lp1 + \frac13\cos\lambda\rp\sin\theta,\; w = \frac13\sin\lambda,\; 0\leq \theta,\lambda < 2\pi \right. \right\}.
\label{eq:Torus}
\end{equation}
This torus is displayed in Figure \ref{fig:Torus}.

\begin{figure}[thb]
\centering
\begin{tabular}{cc}
\subfigure[]{ \includegraphics[width=0.48\textwidth]{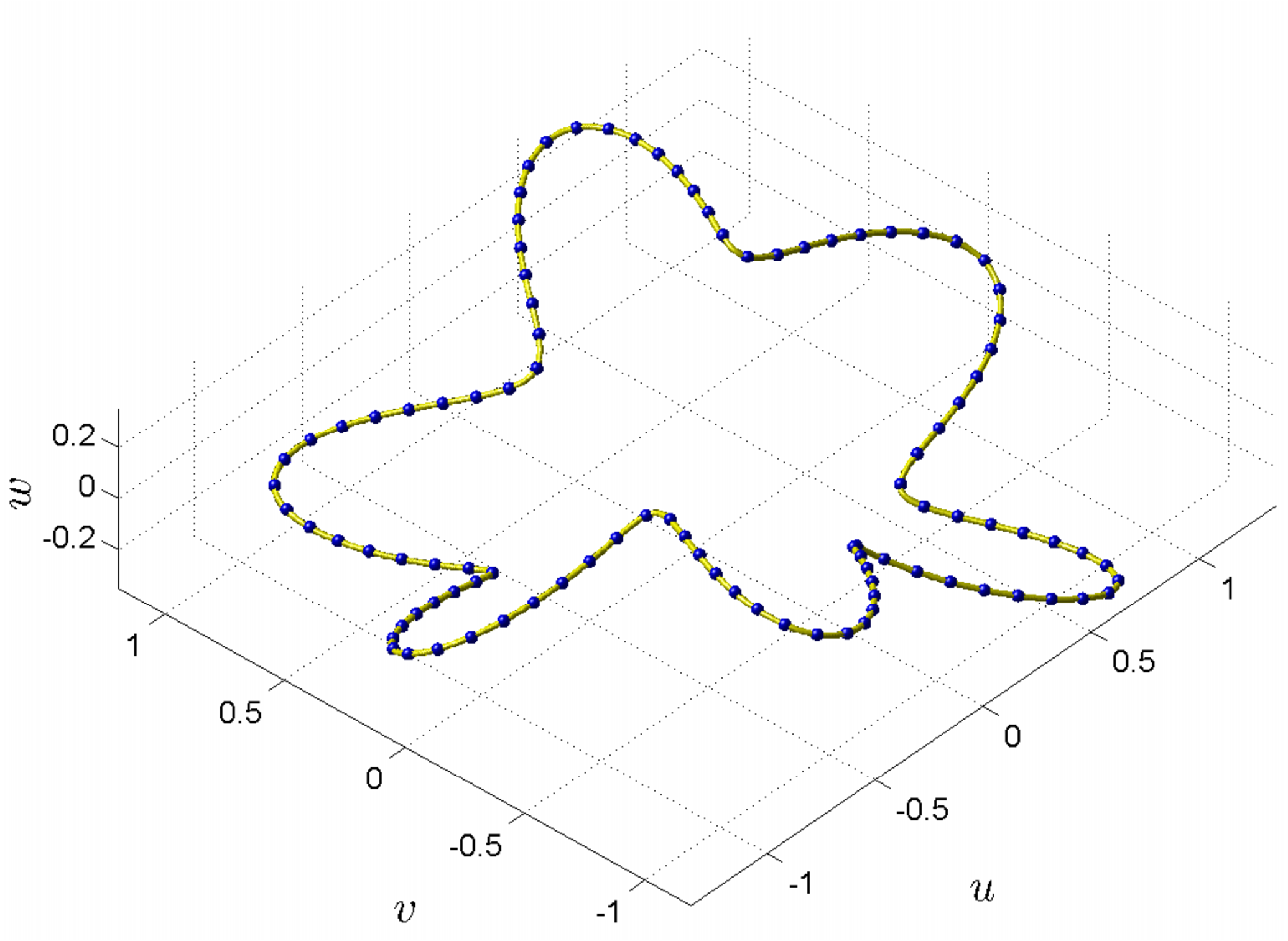} \label{fig:Curve}} &
\subfigure[]{ \includegraphics[width=0.48\textwidth]{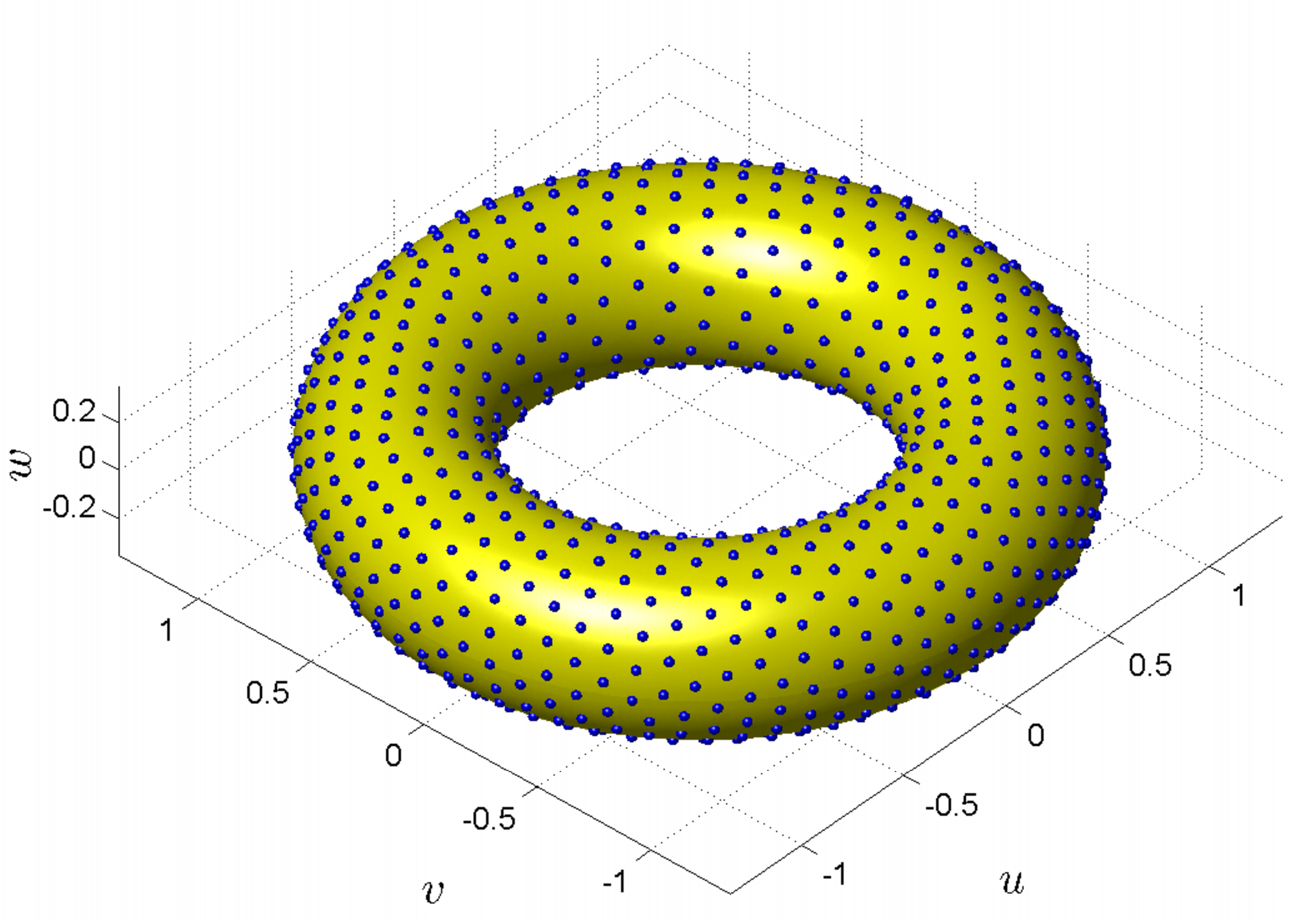} \label{fig:Torus}}
\end{tabular}
\caption{Compact embedded smooth submanifolds used for the numerical experiments.  (a) One-dimensional curve $\M_1$ described parametrically by \eqref{eq:Curve}. (b) Two-dimensional torus $\M_2$  described parametrically by \eqref{eq:Torus}.  Solid spheres mark the locations of the near-minimum Reisz energy interpolation nodes for (a) $N=100$ and (b) $N=1000$.\label{manifolds}}
\end{figure}

To discretize $\M_1$ and $\M_2$, we use a hierarchy of node sets with increasing cardinalities.  For $\M_1$ we use cardinalities of $N=50$, $100$, $200$, $300$, $400$, and $500$, while for $\M_2$ we use $N=500$, $750$, $1000$, $2000$, $3000$, and $4000$.  The node sets for both manifolds are obtained by arranging the nodes so that their Reisz energy (with a power of 2) is near minimal as described in Hardin and Saff's seminal article~\cite{HarinSaff04}.  For $\M_1$, this results in node sets with mesh norms $h_{X,\M_1}$ that decrease like $1/N$, while the mesh norms $h_{X,\M_2}$ for $\M_2$ decrease like $1/\sqrt{N}$.  Additionally, the mesh ratios $\rho_{X,\M_1}$ and $\rho_{X,\M_2}$ remain roughly constant.  The small solid spheres on the curve in Figure \ref{fig:Curve} display the node locations for the $N=100$ node set, while the small spheres on the torus in Figure \ref{fig:Torus} display the $N=1000$ node set.  It is obvious from the latter plot that the nodes are not oriented along any vertices or lines emphasizing the ability of the proposed kernel interpolation technique to handle arbitrary node layouts on a submanifold.

The positive definite kernel we use for constructing the interpolants in the experiments is Wendland's compactly supported RBF
\begin{equation}
\phi_{3,2}(x,y) = \phi_{3,2}(\|x-y\|) = \phi_{3,2}(r) = 
\lp 1-\delr \rp_{+}^6 \lp 3 + 18\delr + 35\lp\delr\rp^2\rp,
\label{eq:Wendland32}
\end{equation}
which is positive definite in $\R^3$ and has 4 continuous derivatives~\cite[\S 9.4]{Wendland:2004}.  Furthermore, the native space of this kernel is known to be $H^{4}(\R^3)$~\cite[p.157]{Wendland:2004}, which means $\tau=4$ in Theorems \ref{theorem: native error estimate} and \ref{theorem: escape error estimate}.  According to Theorem \ref{theorem:restrictednative}, the native spaces on $\M_1$ and $\M_2$ for this kernel are thus $H^{3}(\M_1)$ and $H^{3.5}(\M_2)$. The free parameter $\delta$ is referred to as the support radius and its optimal value depends on numerous factors which are neither easy nor obvious to determine (cf. ~\cite[Ch. 15]{Wendland:2004} or~\cite[Ch. 5]{Iske:2004}).  Since our intention in the present study is only to provide verification of the error estimates presented above, we set $\delta=8/3$ (the maximum distance of any two nodes on either $\M_1$ or $\M_2$) for all the numerical results and leave investigations of selecting $\delta$ for interpolation on submanifolds to a separate study.

One difficulty with verifying error estimates of the type given in Theorems \ref{theorem: native error estimate} and \ref{theorem: escape error estimate} is coming up with explicit forms of target functions belonging to a desired  Sobolev space.  Fortunately, Sobolev spaces can be defined in terms of decay rates of Fourier transforms~\cite[p. 133]{Wendland:2004}.  Using this definition, a straightforward calculation shows any function $\kappa\in \R^d$ whose Fourier transform satisfies
\begin{equation}
\widehat{\kappa}(\xi) \sim (1 + \|\xi\|_{2}^{2})^{-\nu}
\label{eq:SobolevSpaceTarget}
\end{equation}
belongs to every Sobolev space $H^{\beta}(\R^d)$, with $\beta < 2\nu - d/2$.  Thus, we can use functions whose Fourier transforms satisfy (\ref{eq:SobolevSpaceTarget}) with $\nu = (\beta + d/2)/2$ as target functions since they will then be in $H^{s}(\M)$ for all $s < \beta - (d-k)/2$.  

A well-known class of functions that satisfy (\ref{eq:SobolevSpaceTarget}) (with strict equality) are the Mat\'{e}rn kernels~\cite{Matern:1986SpatialVariation} or Sobolev splines and are defined as
\begin{equation*}
\kappa_{\nu}(x,y) = \kappa_{\nu}(\|x - y\|) = \kappa_{\nu}(r) = \frac{2^{1-(\nu-d/2)}}{\Gamma(\nu-d/2)}r^{\nu-d/2}K_{\nu-d/2}(r),\; \nu > d/2
\end{equation*}
where $K_{\nu-d/2}$ corresponds to the $K$-Bessel function of order $\nu-d/2$.  To generate interesting target functions for verifying Theorems \ref{theorem: native error estimate} and \ref{theorem: escape error estimate}, we use linear combinations of these Mat\'{e}rn kernels as follows.  Let $X = \{x_1,x_2,\ldots,x_m\}$ be some set of distinct points on the submanifold under consideration (either $\M_1$ or $\M_2$), then the target function is given by
\begin{equation}
f_{\beta}(x) = \sum_{j=1}^{m} c_j \kappa_{\frac12(\beta + d/2)}(\|x - x_j\|).
\label{eq:TargetFunction}
\end{equation}
The coefficients $c_1,c_2,\ldots,c_m$ are determined by the requiring $f_\beta$ interpolate the following function at the points in $X$:
\begin{equation}
p(x) = \frac18\lp u^5 - 10u^3 v^2 + 5 u v^4\rp \lp u^2 + v^2 - 60 w^2 \rp,
\label{eq:SmoothTargetFunction}
\end{equation}
where $u$, $v$, and $w$ are the components of $x$. For $\M_1$, we use $m=25$ quasi-minimum Reisz energy points for $X$, while for $\M_2$ we use $n=100$ quasi-minimum Reisz energy points.  Plots of the target function for $\beta=4$ are displayed in Figure \ref{fig:TargetFunctions}(a) and (b) for $\M_1$ and $\M_2$, respectively.  Note that $f_{\beta}$ is \emph{not} the kernel interpolant to be compared against the theoretical error estimates, rather it is the form of the target functions to be interpolated.

\begin{figure}[t]
\centering
\begin{tabular}{cc}
\subfigure[]{ \includegraphics[width=0.42\textwidth]{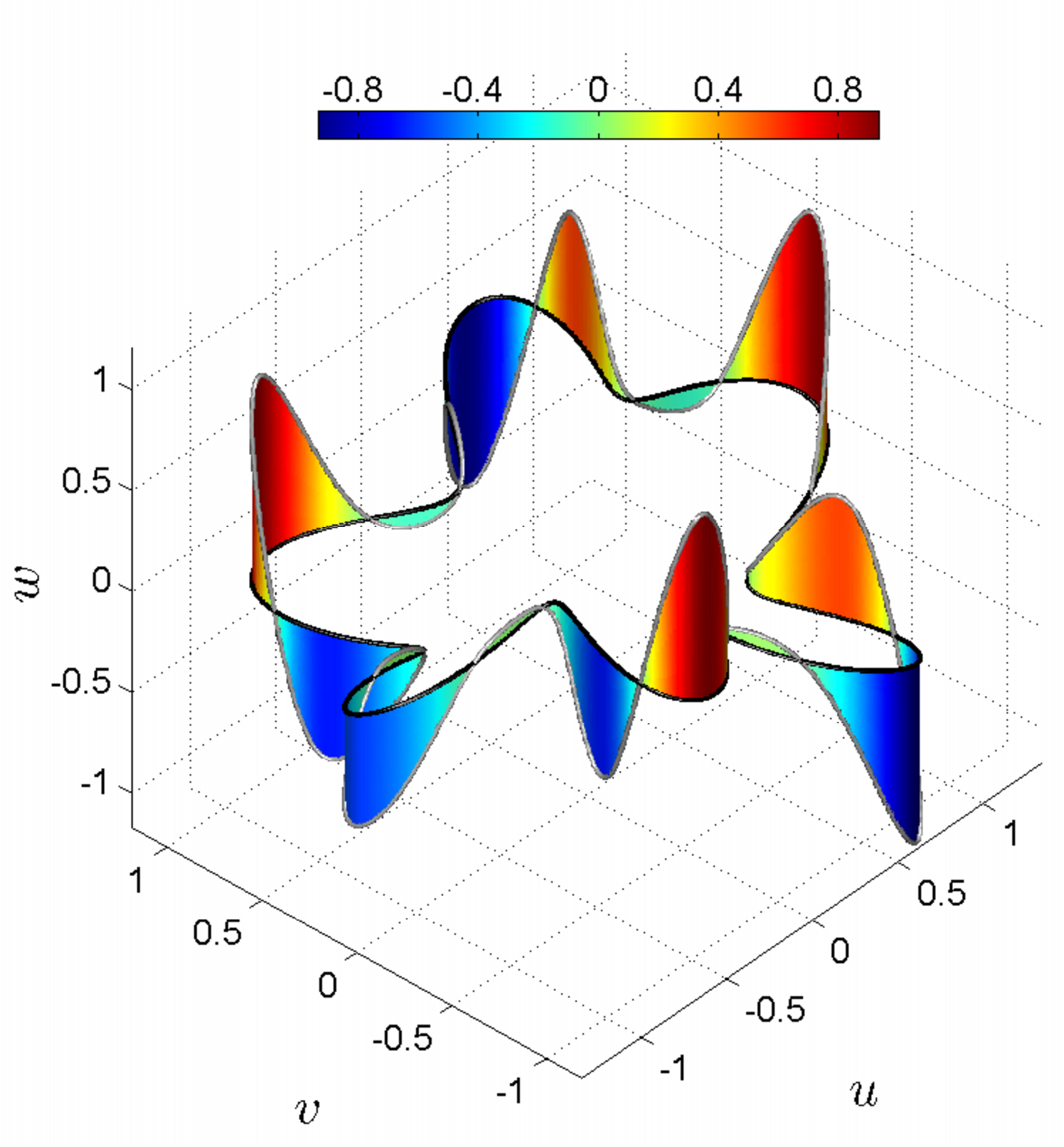} \label{fig:CurveTarget}} &
\subfigure[]{ \includegraphics[width=0.54\textwidth]{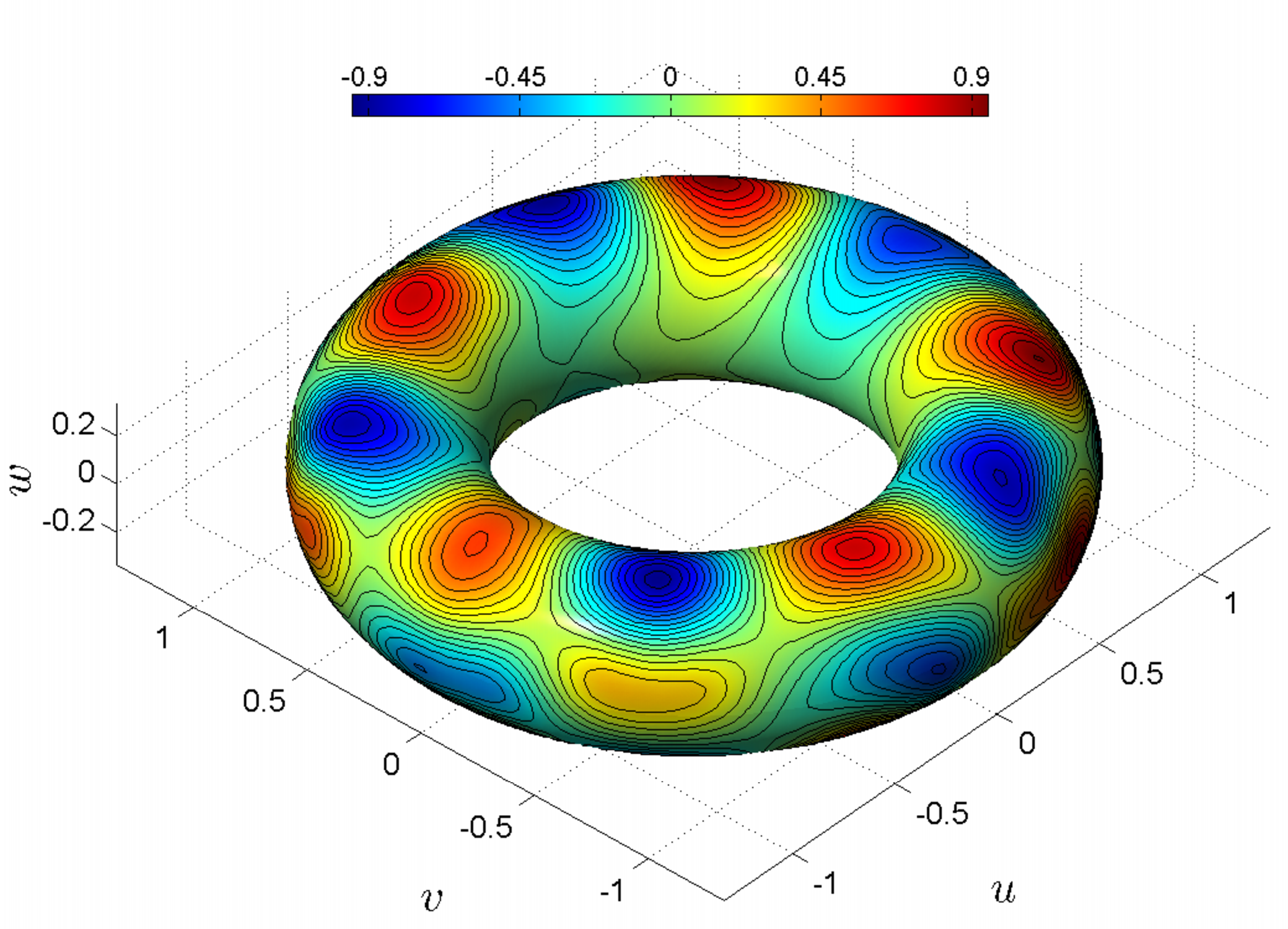} \label{fig:TorusTarget}}
\end{tabular}
\caption{Visualization of two of the target functions for the numerical experiments.  (a) The target function (\ref{eq:TargetFunction}) for $\M_1$ with $m=25$ and $\beta=4$ displayed as the height above the submanifold (i.e. $(u,v,w+f_{4}(x))$) together with colors indicating the values of the function. (b) The target function for $\M_2$ with $m=100$ and $\beta=4$ with colors corresponding to the values of the target function on the submanifold and black lines corresponding to contours of the function. \label{fig:TargetFunctions}}
\end{figure}

In the results below, the errors are measured by evaluating the kernel interpolants and the target functions at a much denser set of points that sufficiently cover the manifolds.  For $\M_1$, $P = 3000$ evaluation points are used, while $P=24,300$ points are used for $\M_2$.  We approximate the (relative) $L_2(\M)$-norms of the errors by approximating the surface integrals over the manifolds using a midpoint-type rule.  We use the following abuse of notation to denote the approximate $L_2(\M)$-norm:
\begin{equation*}
\|f\|_{L_2(\M)} := \lp\int_{\M} [f(x)]^{2} dx\rp^{1/2} \approx \lp\sum_{i=1}^P w_i [f(x_i)]^2\rp^{1/2} := \|f\|_{\ell_2(\M)},
\end{equation*}
where $\{w_i\}_{i=1}^{P}$ are quadrature weights for the evaluation points $\{x_i\}_{i=1}^{P}$ on the manifold.  We also meausure the max-norm errors over the manifolds.  We use the standard definition for this norm and denote it with the standard notation of $\ell_{\infty}(\M)$.

\begin{figure}[t!]
\centering
\begin{tabular}{cc}
\subfigure[$\M_1$, $\beta=4$]{ \includegraphics[width=0.45\textwidth]{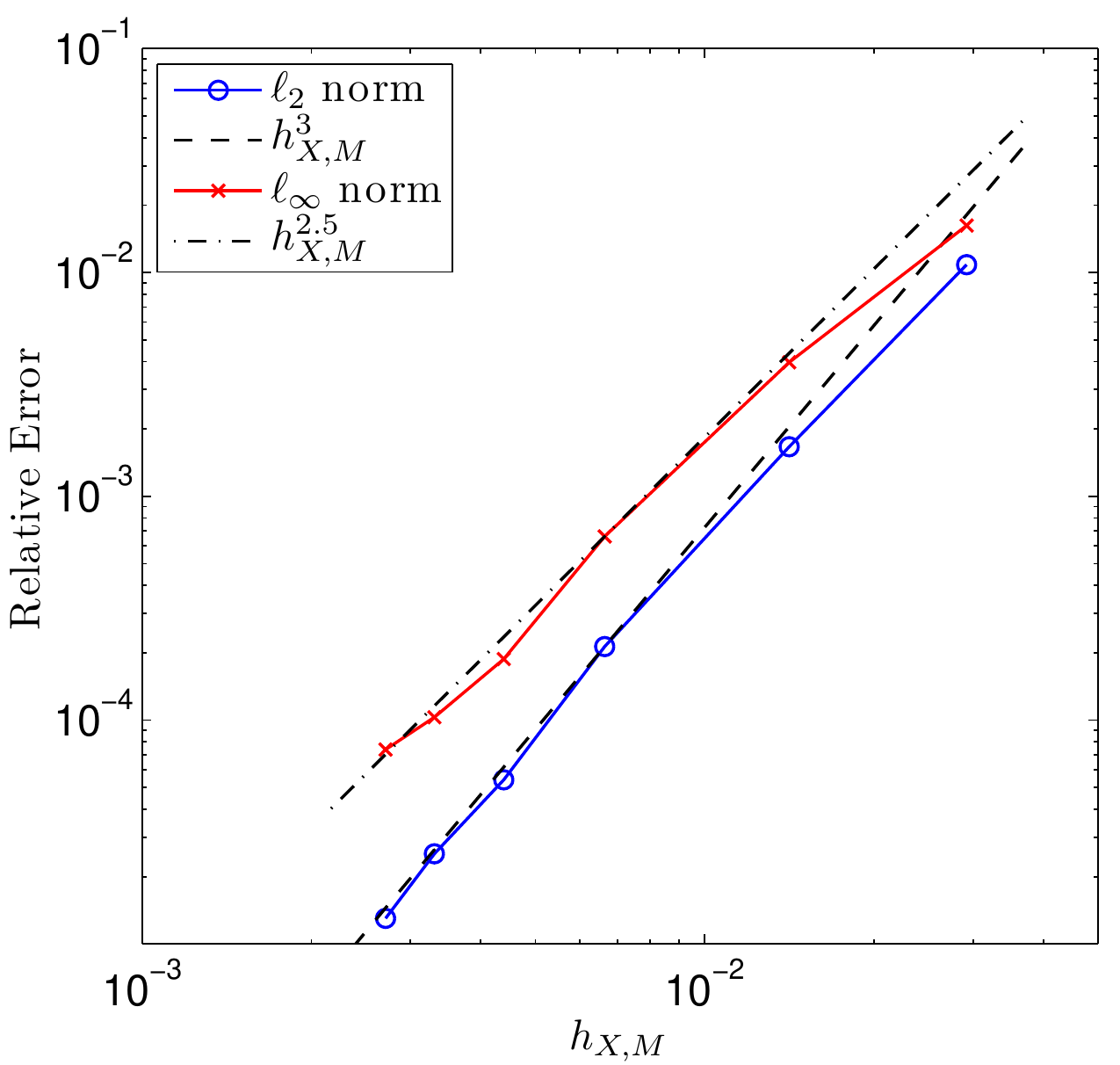} \label{fig:ErrorM1Beta5}} &
\subfigure[$\M_2$, $\beta=4$]{ \includegraphics[width=0.45\textwidth]{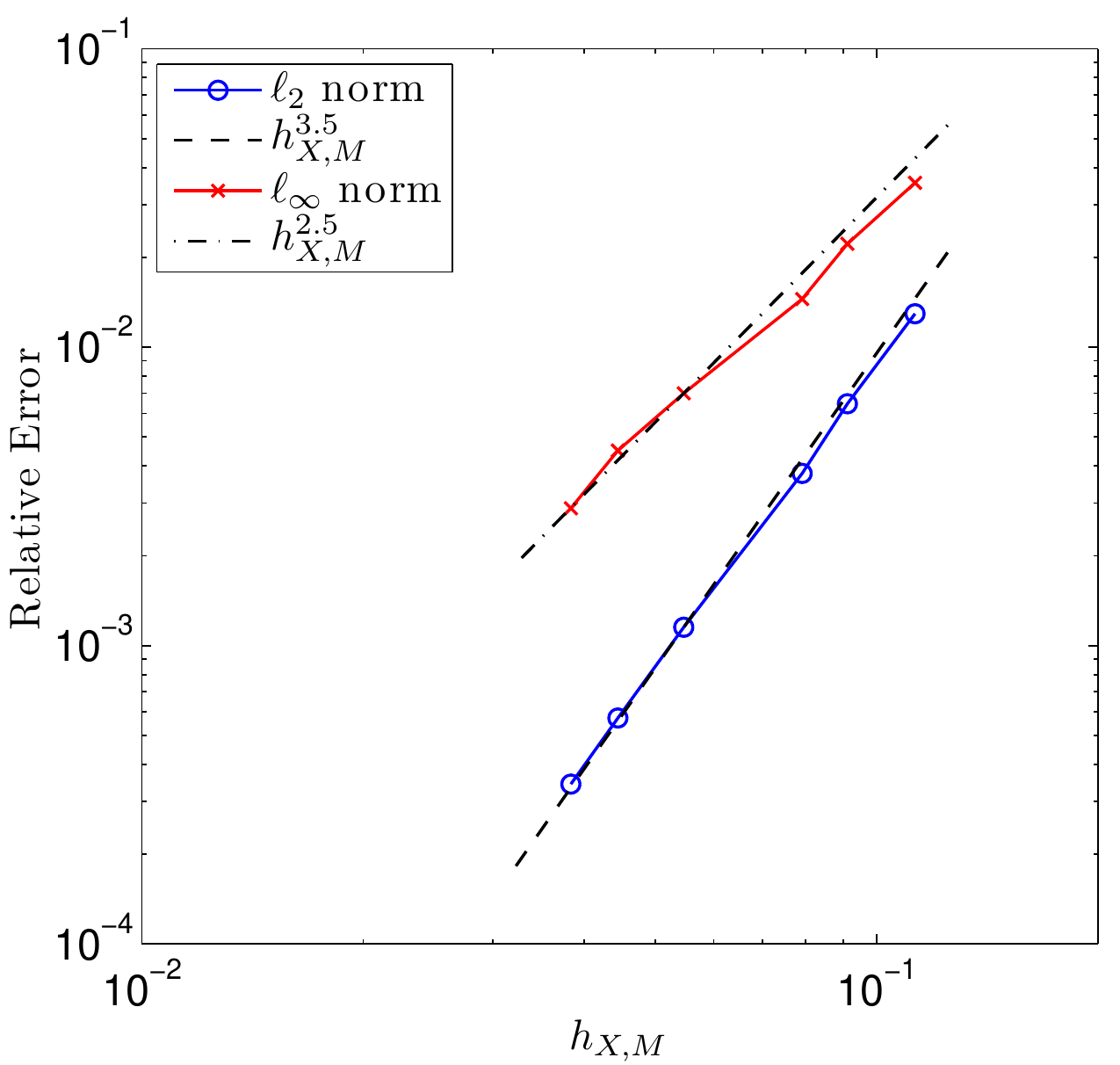} \label{fig:ErrorM2Beta5}} \\
\subfigure[$\M_1$, $\beta=3.5$]{ \includegraphics[width=0.45\textwidth]{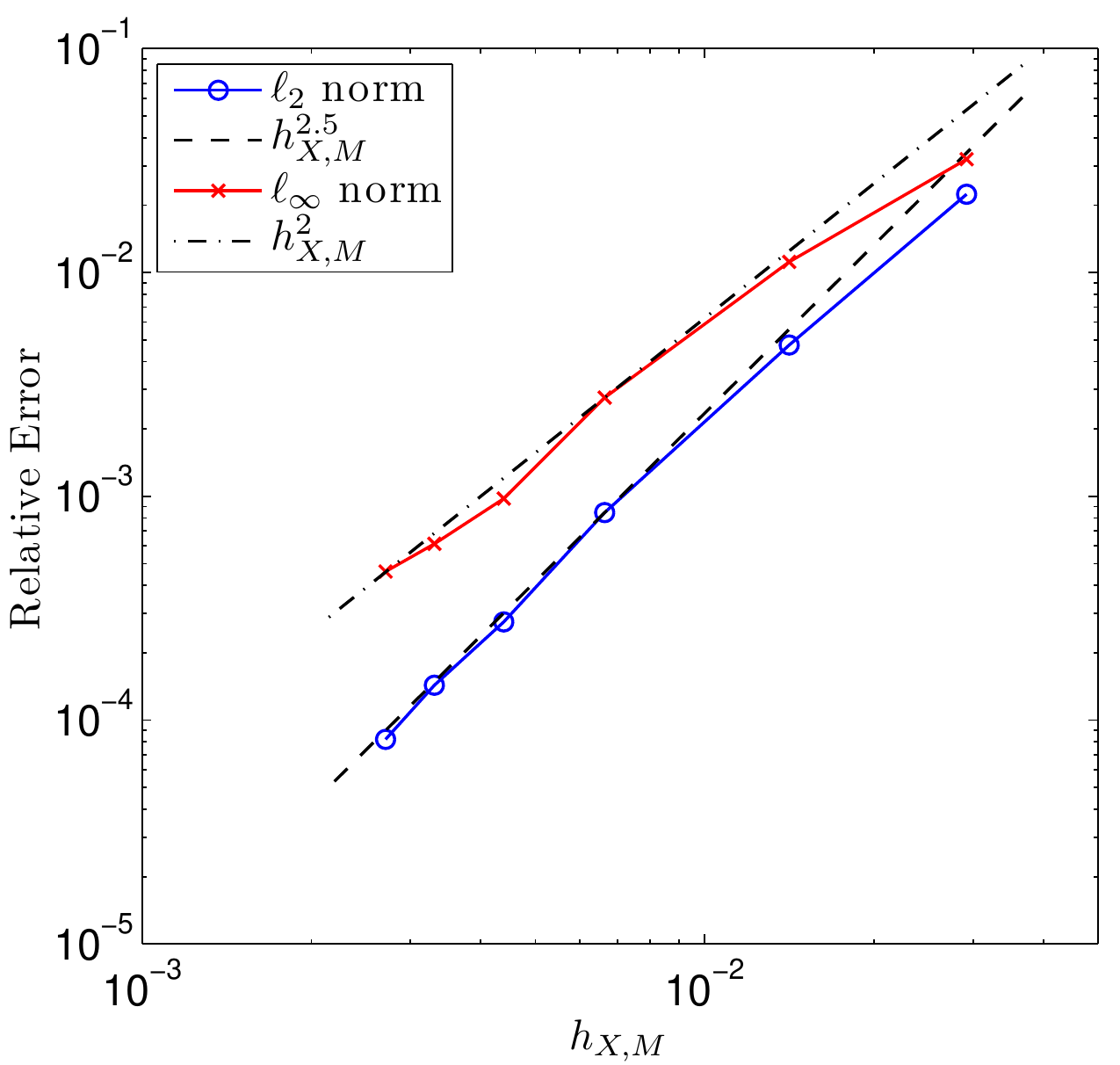} \label{fig:ErrorM1Beta3_5}} &
\subfigure[$\M_2$, $\beta=3.5$]{ \includegraphics[width=0.45\textwidth]{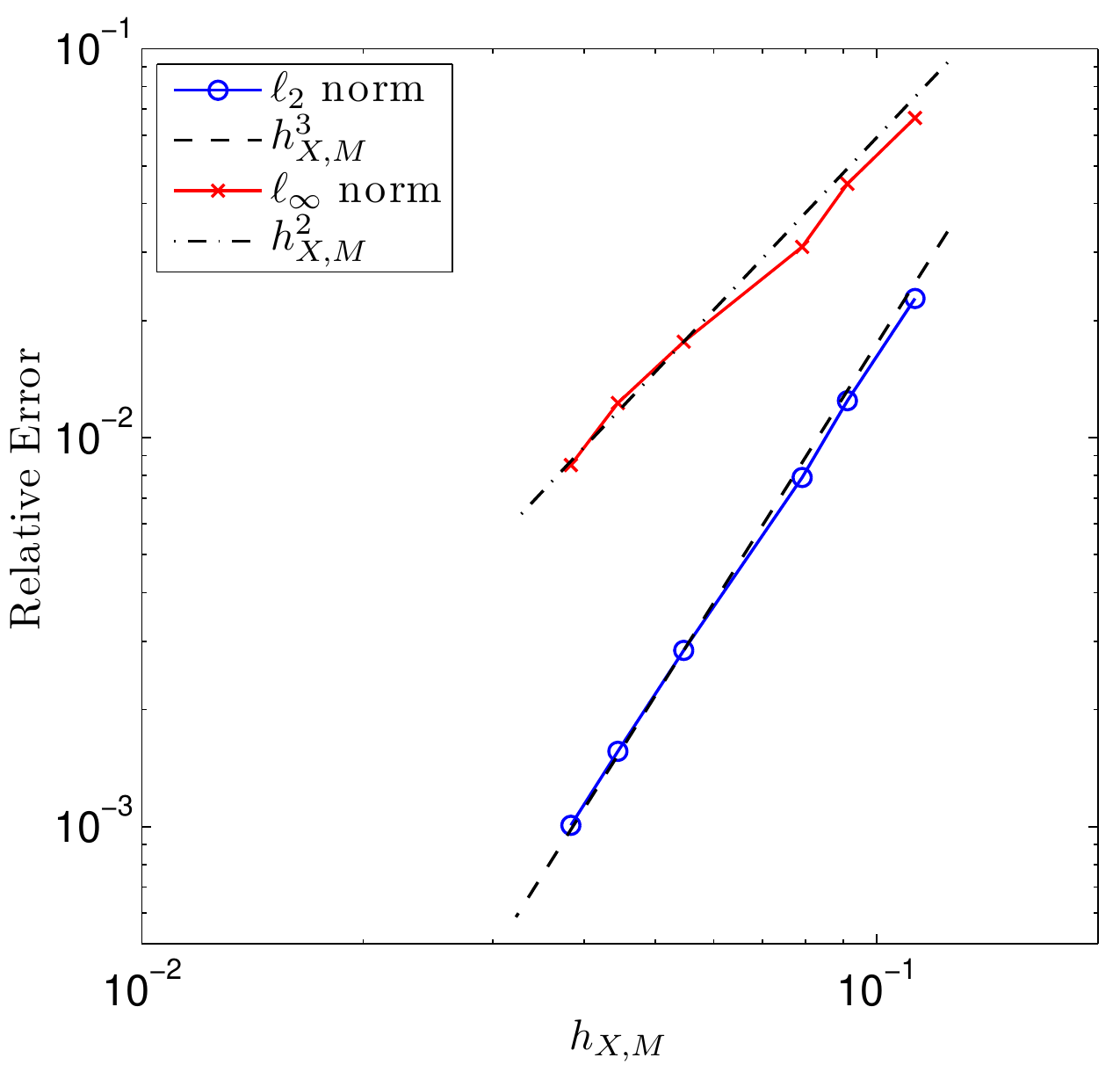} \label{fig:ErrorM2Beta3_5}}
\end{tabular}
\caption{Relative errors in reconstructing the different target functions $f_{\beta}$ in (\ref{eq:TargetFunction}) on the two different submanifolds with the kernel (\ref{eq:Wendland32}).  Plots in (a) and (c) are for $\M_1$ and plots in (b) and (d) are for $M_2$.  Lines marked with open circles in each figure correspond to the relative $\ell_2(\M)$ error and lines marked with x's are for the relative $\ell_{\infty}(\M)$ errors.  In (a) and (b) the dashed and dashed-dotted lines correspond to the predicted error estimates from Theorem \ref{theorem: native error estimate} for $\ell_2(\M)$ and $\ell_{\infty}(\M)$, respectively, while these lines in (c) and (d) correspond to the predicted estimates from Theorem \ref{theorem: escape error estimate}. \label{fig:Error}}
\end{figure}

For the first numerical experiment, we set set $\beta=4$ in (\ref{eq:TargetFunction}) for $\M_1$ and $\M_2$.  Since the native space for the kernel (\ref{eq:Wendland32}) is $H^{4}(\R^3)$, we expect the estimates from Theorem \ref{theorem: native error estimate} to apply for these target functions.  According to this theorem, the $\ell_2$ errors for $\M_1$ and $\M_2$ should decrease like $h_{X,\M_1}^3$ and $h_{X,\M_2}^{3.5}$, while the $\ell_{\infty}$ errors should decrease like $h_{X,\M_1}^{2.5}$ and $h_{X,\M_2}^{2.5}$.  Figures \ref{fig:Error}(a) and (b) display the computed relative errors versus the mesh norm for $f_{4}$ together with the predicted estimates.  The figures show good agreement between the numerical and theoretical results. 

In the second experiment, we set $\beta = 3.5$ in (\ref{eq:TargetFunction}) for $\M_1$ and $\M_2$, which makes the target functions rougher than the native space of the kernel and means the estimates from Theorem \ref{theorem: escape error estimate} will apply.  Assuming the mesh ratio is roughly constant (which is true for our experiments), this theorem predicts that the $\ell_2$ errors for $\M_1$ and $\M_2$ should decrease like $h_{X,\M_1}^{2.5}$ and $h_{X,\M_2}^{3}$, while the $\ell_{\infty}$ errors should decrease like $h_{X,\M_1}^{2}$ and $h_{X,\M_2}^{2}$.  Similar to the previous experiment, Figures \ref{fig:Error}(c) and (d) display the computed relative errors for these target functions  together with the predicted estimates.  Good agreement between the numerical and theoretical results is again displayed.

As a final numerical experiment, we verify the ``doubling'' estimates from  Corollary \ref{theorem: doubling trick}.  For the target function we use (\ref{eq:SmoothTargetFunction}) directly.  This function is $C^{\infty}(\R^3)$ and is thus much smoother than the native space of the kernel (\ref{eq:Wendland32}).  According to Corollary \ref{theorem: doubling trick}, the $\ell_2$ errors for $\M_1$ and $\M_2$ should decrease like $h_{X,\M_1}^6$ and $h_{X,\M_2}^{7}$ for this target function and kernel. Figure \ref{fig:SmoothError}(a) and (b) displays the results for $\M_1$ and $\M_2$, respectively.  We again find good agreement between the numerical and theoretical results.  The results from the figure also indicate that there does not appear to be a reduction in the estimates for the $\ell_{\infty}$ error for this case.

\begin{figure}[thb]
\centering
\begin{tabular}{cc}
\subfigure[$\M_1$]{\includegraphics[width=0.45\textwidth]{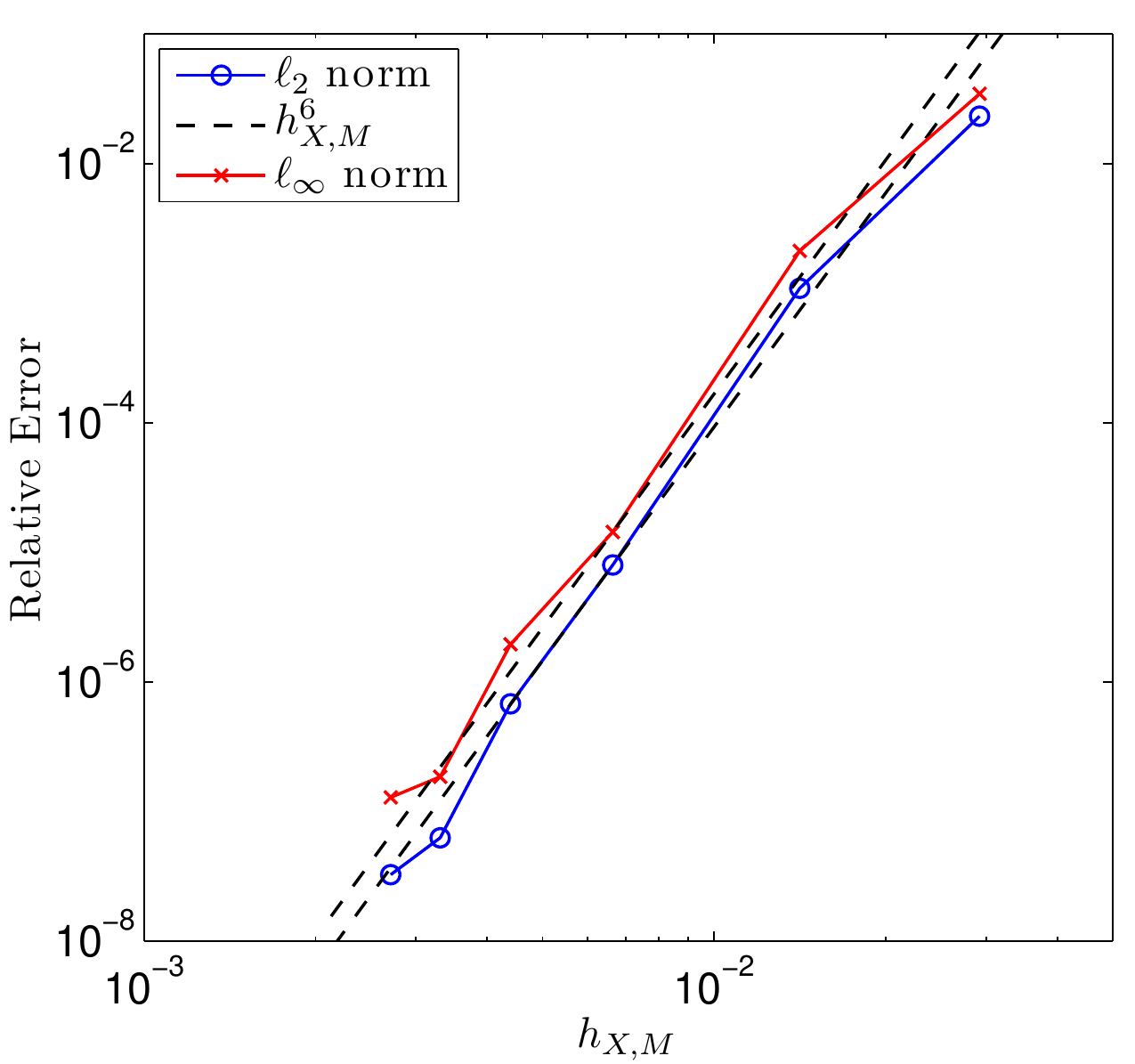} \label{fig:ErrorM1BetaInf}} &
\subfigure[$\M_2$]{\includegraphics[width=0.45\textwidth]{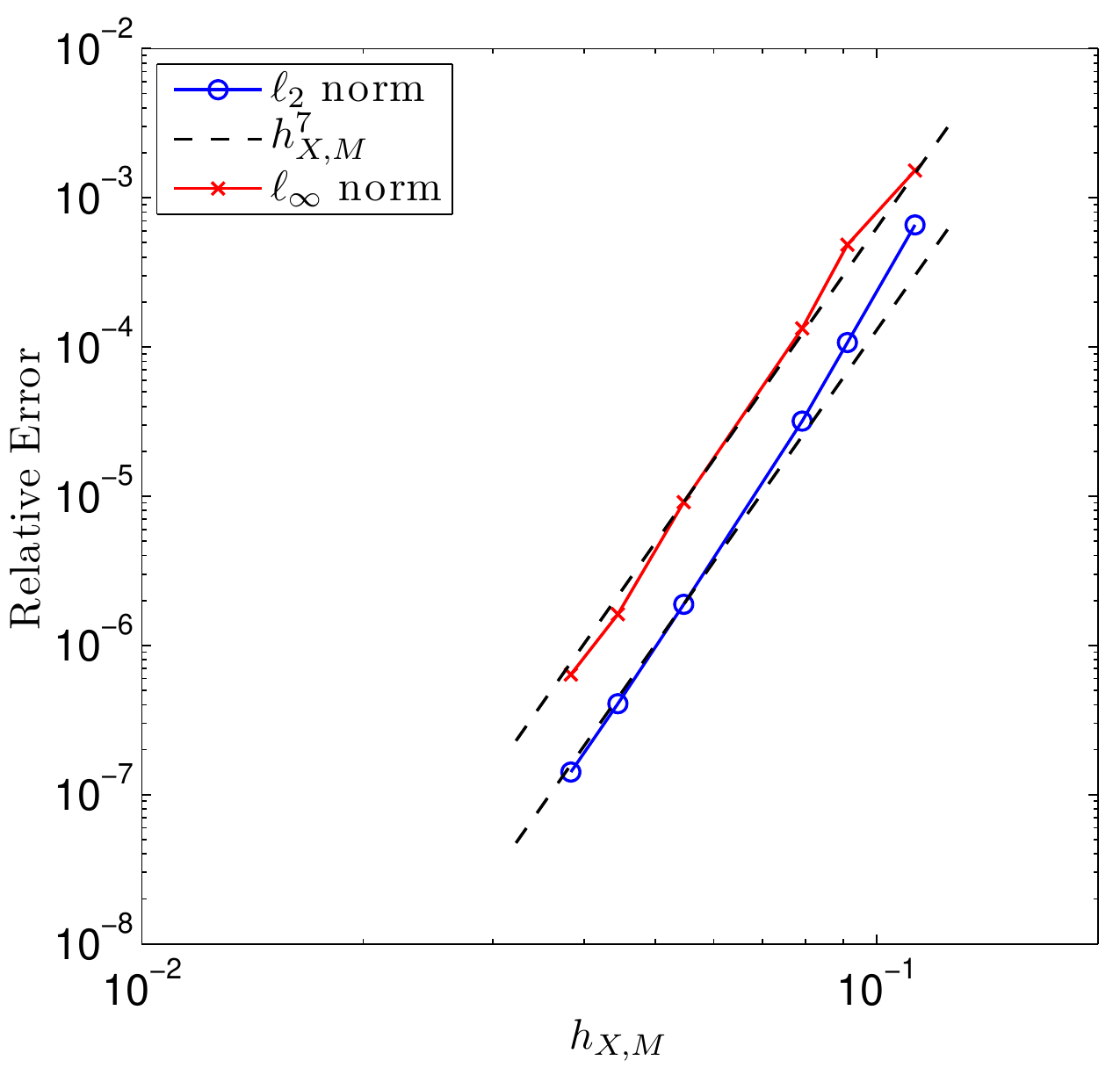} \label{fig:ErrorM2BetaInf}} \\
\end{tabular}
\caption{Relative errors in reconstructing the target function (\ref{eq:SmoothTargetFunction}) on (a) $\M_1$ and (b) $\M_2$ with the kernel (\ref{eq:Wendland32}).  The dashed lines in the figure correspond to the predicted error estimates for the $\ell_2(\M)$ error from Corollary \ref{theorem: doubling trick}. \label{fig:SmoothError}}
\end{figure}

%%%%%%%%%%%%%%%%%%%%%%%%%%%%%%%%%%%%%%%%%%%%%%%%%%%%%%%%
\appendix{ \bf Acknowledgments.}
%%%%%%%%%%%%%%%%%%%%%%%%%%%%%%%%%%%%%%%%%%%%%%%%%%%%%%%%
We wish to express our deep gratitude to Drs. Douglas Hardin and Edward Saff and Ms. Ayla Gafni, all from Vanderbilt University, for providing us with the near-minimum Reisz energy points for the torus used in the numerical experiments.

\bibliographystyle{abbrv}
\bibliography{refs}

\begin{thebibliography}{10}

\bibitem{Adalsteinsson:2003}
D.~Adalsteinsson and J.~A. Sethian.
\newblock Transport and diffusion of material quantities on propagating
  interfaces via level set methods.
\newblock {\em J. Comput. Phys.}, 185(1):271--288, 2003.

\bibitem{Adams:1975SobolevSpaces}
R.~A. Adams.
\newblock {\em Sobolev spaces}.
\newblock Academic Press [A subsidiary of Harcourt Brace Jovanovich,
  Publishers], New York-London, 1975.
\newblock Pure and Applied Mathematics, Vol. 65.

\bibitem{Alfeld:1996}
P.~Alfeld, M.~Neamtu, and L.~L. Schumaker.
\newblock Fitting scattered data on sphere-like surfaces using spherical
  splines.
\newblock {\em J. Comput. Appl. Math.}, 73(1-2):5--43, 1996.

\bibitem{Bajaj:1995}
C.~L. Bajaj, F.~Bernardini, and G.~Xu.
\newblock Automatic reconstruction of surfaces and scalar fields from {3D}
  scans.
\newblock In {\em SIGGRAPH '95: Proceedings of the 22nd annual conference on
  Computer graphics and interactive techniques}, pages 109--118, New York, NY,
  USA, 1995. ACM.

\bibitem{BarnhillOu:1990}
R.~E. Barnhill and H.~S. Ou.
\newblock Surfaces defined on surfaces.
\newblock {\em Comput. Aided Geom. Des.}, 7(1-4):323--336, 1990.

\bibitem{BelkinNiyogi2004:LearningOnRiemannianManifolds}
M.~Belkin and P.~Niyogi.
\newblock Semi-supervised learning on riemannian manifolds.
\newblock {\em Mach. Learn.}, 56:209--239, 2004.

\bibitem{BesovEtAl1979:IntegralRep}
O.~V. Besov, V.~P. Il$'$in, and S.~M. Nikol$'$ski{\u\i}.
\newblock {\em Integral representations of functions and imbedding theorems.
  {V}ol. {II}}.
\newblock V. H. Winston \& Sons, Washington, D.C., 1979.
\newblock Scripta Series in Mathematics, Edited by Mitchell H. Taibleson.

\bibitem{Calhoun:2009}
D.~A. Calhoun and C.~Helzel.
\newblock A finite volume method for solving parabolic equations on logically
  {C}artesian curved surface meshes.
\newblock {\em SIAM Journal on Scientific Computing}, 31(6):4066--4099, 2009.

\bibitem{Clarenz:2000}
U.~Clarenz, U.~Diewald, and M.~Rumpf.
\newblock Anisotropic geometric diffusion in surface processing.
\newblock In {\em VISUALIZATION '00: Proceedings of the 11th IEEE Visualization
  2000 Conference (VIS 2000)}, Washington, DC, USA, 2000. IEEE Computer
  Society.

\bibitem{Davydov:2007}
O.~Davydov and L.~L. Schumaker.
\newblock Scattered data fitting on surfaces using projected powell-sabin
  splines.
\newblock In {\em Proceedings of the 12th IMA international conference on
  Mathematics of surfaces XII}, pages 138--153, Berlin, Heidelberg, 2007.
  Springer-Verlag.

\bibitem{Desbrun:1999}
M.~Desbrun, M.~Meyer, P.~Schr\"{o}der, and A.~H. Barr.
\newblock Implicit fairing of irregular meshes using diffusion and curvature
  flow.
\newblock In {\em SIGGRAPH '99: Proceedings of the 26th annual conference on
  Computer graphics and interactive techniques}, pages 317--324, New York, NY,
  USA, 1999. ACM Press/Addison-Wesley Publishing Co.

\bibitem{FilbirErb2008:PDFsOnCompactGroups}
W.~Erb and F.~Filbir.
\newblock Approximation by positive definite functions on compact groups.
\newblock {\em Numer. Funct. Anal. Optim.}, 29(9-10):1082--1107, 2008.

\bibitem{Ferreira2009}
J.~C. Ferreira and V.~A. Menegatto.
\newblock Eigenvalues of integral operators defined by smooth positive definite
  kernels.
\newblock {\em Integral Equations Operator Theory}, 64(1):61--81, 2009.

\bibitem{FilbirSchmid2008:SO3Stability}
F.~Filbir and D.~Schmid.
\newblock Stability results for approximation by positive definite functions on
  {${\rm SO}(3)$}.
\newblock {\em J. Approx. Theory}, 153(2):170--183, 2008.

\bibitem{FlyerWright09}
N.~Flyer and G.~B. Wright.
\newblock A radial basis function method for the shallow water equations on a
  sphere.
\newblock {\em Proc. Roy. Soc. A}, 465:1949--1976, 2009.

\bibitem{Foley:1990}
T.~A. Foley, D.~A. Lane, G.~M. Nielson, R.~Franke, and H.~Hagen.
\newblock Interpolation of scattered data on closed surfaces.
\newblock {\em Comput. Aided Geom. Des.}, 7(1-4):303--312, 1990.

\bibitem{FuselierWright:2009VectorDecomposition}
E.~J. Fuselier and G.~B. Wright.
\newblock Stability and error estimates for vector field interpolation and
  decomposition on the sphere with rbfs.
\newblock {\em SIAM J. Numer. Anal.}, 47(5):3213--3239, 2009.

\bibitem{HangelbroekNarcWard2010:ManifoldKernelsI}
T.~Hangelbroek, F.~Narcowich, and J.~Ward.
\newblock Kernel approximation on manifolds {I}: Bounding the lebesgue
  constant.
\newblock To appear in \emph{SIAM J. on Math. Anal.}, (2010),
  http://arxiv.org/abs/0909.0033v2.

\bibitem{HarinSaff04}
D.~P. Hardin and E.~B. Saff.
\newblock Discretizing manifolds via minimum energy points.
\newblock {\em Notices Amer. Math. Soc.}, 51:1186--1194, 2004.

\bibitem{Iske:2004}
A.~Iske.
\newblock {\em Multiresolution Methods in Scattered Data Modelling}.
\newblock Springer-Verlag, Heidelberg, 2004.

\bibitem{Lee2003:smoothmanifolds}
J.~M. Lee.
\newblock {\em Introduction to smooth manifolds}, volume 218 of {\em Graduate
  Texts in Mathematics}.
\newblock Springer-Verlag, New York, 2003.

\bibitem{LevesleyRagozin:2007RBFsHomogenousManifolds}
J.~Levesley and D.~L. Ragozin.
\newblock Radial basis interpolation on homogeneous manifolds: convergence
  rates.
\newblock {\em Adv. Comput. Math.}, 27(2):237--246, 2007.

\bibitem{Lions:1961nonhomogeneous}
J.-L. Lions and E.~Magenes.
\newblock Problemi ai limiti non omogenei. {III}.
\newblock {\em Ann. Scuola Norm. Sup. Pisa (3)}, 15:41--103, 1961.

\bibitem{MadychNelson1988:MultivariateInterpolation}
W.~R. Madych and S.~A. Nelson.
\newblock Multivariate interpolation and conditionally positive definite
  functions.
\newblock {\em Approx. Theory Appl.}, 4(4):77--89, 1988.

\bibitem{Matern:1986SpatialVariation}
B.~Mat{\'e}rn.
\newblock {\em Spatial variation}, volume~36 of {\em Lecture Notes in
  Statistics}.
\newblock Springer-Verlag, Berlin, second edition, 1986.
\newblock With a Swedish summary.

\bibitem{MORTON_NEAMTU:2002}
T.~M. Morton and M.~Neamtu.
\newblock Error bounds for solving pseudodifferential equations on spheres by
  collocation with zonal kernels.
\newblock {\em J. Approx. Theory}, 144(2):242--268, 2002.

\bibitem{Narc1995:kernelsonmanifolds}
F.~J. Narcowich.
\newblock Generalized {H}ermite interpolation and positive definite kernels on
  a {R}iemannian manifold.
\newblock {\em J. Math. Anal. Appl.}, 190(1):165--193, 1995.

\bibitem{Narc2005:RecentErrorSurvey}
F.~J. Narcowich.
\newblock Recent developments in error estimates for scattered-data
  interpolation via radial basis functions.
\newblock {\em Numer. Algorithms}, 39(1-3):307--315, 2005.

\bibitem{NarcSunWard2007:SBFRBFconnection}
F.~J. Narcowich, X.~Sun, and J.~D. Ward.
\newblock Approximation power of {RBF}s and their associated {SBF}s: a
  connection.
\newblock {\em Adv. Comput. Math.}, 27(1):107--124, 2007.

\bibitem{NarcSunWard:2007direct}
F.~J. Narcowich, X.~Sun, J.~D. Ward, and H.~Wendland.
\newblock Direct and inverse {S}obolev error estimates for scattered data
  interpolation via spherical basis functions.
\newblock {\em Found. Comput. Math.}, 7(3):369--390, 2007.

\bibitem{NarcWard1994:hermiteinterpolation}
F.~J. Narcowich and J.~D. Ward.
\newblock Generalized {H}ermite interpolation via matrix-valued conditionally
  positive definite functions.
\newblock {\em Math. Comp.}, 63(208):661--687, 1994.

\bibitem{NarcWardWendland:2005ScatteredZeros}
F.~J. Narcowich, J.~D. Ward, and H.~Wendland.
\newblock Sobolev bounds on functions with scattered zeros, with applications
  to radial basis function surface fitting.
\newblock {\em Math. Comp.}, 74(250):743--763 (electronic), 2005.

\bibitem{NarcWardWend:2006bernstein}
F.~J. Narcowich, J.~D. Ward, and H.~Wendland.
\newblock Sobolev error estimates and a {B}ernstein inequality for scattered
  data interpolation via radial basis functions.
\newblock {\em Constr. Approx.}, 24(2):175--186, 2006.

\bibitem{Neamtu:2002}
M.~Neamtu.
\newblock Splines on surfaces.
\newblock In {\em Handbook on {CAGD}}, pages 229--253. North-Holland,
  Amsterdam, 2002.

\bibitem{Novak:2007}
I.~L. Novak, F.~Gao, Y.-S. Choi, D.~Resasco, J.~C. Schaff, and B.~M.
  Slepchenko.
\newblock Diffusion on a curved surface coupled to diffusion in the volume:
  Application to cell biology.
\newblock {\em J. Comput. Phys.}, 226(2):1271--1290, 2007.

\bibitem{LevesleyRagozin2007:DensityOfZonalKernels}
D.~L. Ragozin and J.~Levesley.
\newblock The density of translates of zonal kernels on compact homogeneous
  spaces.
\newblock {\em J. Approx. Theory}, 103(2):252--268, 2000.

\bibitem{Sbalzarini:2006}
I.~F. Sbalzarini, A.~Hayer, A.~Helenius, and P.~Koumoutsakos.
\newblock Simulations of (an)isotropic diffusion on curved biological surfaces.
\newblock {\em Biophys. J.}, 90:878--885, 2006.

\bibitem{SCHA:1999}
R.~Schaback.
\newblock Improved error bounds for scattered data interpolation by radial
  basis functions.
\newblock {\em Math. Comput.}, 68(225):201--216, 1999.

\bibitem{Schaback1999:NativeSpacesI}
R.~Schaback.
\newblock Native {H}ilbert spaces for radial basis functions. {I}.
\newblock In {\em New developments in approximation theory ({D}ortmund, 1998)},
  volume 132 of {\em Internat. Ser. Numer. Math.}, pages 255--282.
  Birkh\"auser, Basel, 1999.

\bibitem{Schaback2000:NativeSpacesII}
R.~Schaback.
\newblock A unified theory of radial basis functions. {N}ative {H}ilbert spaces
  for radial basis functions. {II}.
\newblock {\em J. Comput. Appl. Math.}, 121(1-2):165--177, 2000.
\newblock Numerical analysis in the 20th century, Vol. I, Approximation theory.

\bibitem{Schwartz:2005}
P.~Schwartz, D.~Adalsteinsson, P.~Colella, A.~P. Arkin, and M.~Onsum.
\newblock Numerical computation of diffusion on a surface.
\newblock {\em Proc. Nat. Acad. Sci.}, 102(32):11151--11156, 2005.

\bibitem{Stam:2003}
J.~Stam.
\newblock Flows on surfaces of arbitrary topology.
\newblock In {\em SIGGRAPH '03: ACM SIGGRAPH 2003 Papers}, pages 724--731, New
  York, NY, USA, 2003. ACM.

\bibitem{Stewart1976:PDFSurvey}
J.~Stewart.
\newblock Positive definite functions and generalizations, an historical
  survey.
\newblock {\em Rocky Mountain J. Math.}, 6(3):409--434, 1976.

\bibitem{Wendland:2004}
H.~Wendland.
\newblock {\em Scattered data approximation}, volume~17 of {\em Cambridge
  Monographs on Applied and Computational Mathematics}.
\newblock Cambridge University Press, Cambridge, 2005.

\bibitem{Wu1992:hermiteRBF}
Z.~M. Wu.
\newblock Hermite-{B}irkhoff interpolation of scattered data by radial basis
  functions.
\newblock {\em Approx. Theory Appl.}, 8(2):1--10, 1992.

\bibitem{zuCastellFilbir2005:RestrictedRBFs}
W.~zu~Castell and F.~Filbir.
\newblock Radial basis functions and corresponding zonal series expansions on
  the sphere.
\newblock {\em J. Approx. Theory}, 134(1):65--79, 2005.

\bibitem{ZuCastellFilbir2007:MotionGroups}
W.~zu~Castell and F.~Filbir.
\newblock Strictly positive definite functions on generalized motion groups.
\newblock In {\em Algorithms for approximation}, pages 349--357. Springer,
  Berlin, 2007.

\end{thebibliography}

\end{document}